%% file: article.tex
\RequirePackage{fix-cm}
\documentclass[twocolumn]{svjour3}          
\smartqed  

\usepackage[utf8]{inputenc}
\usepackage[russian, english]{babel}
\DeclareUnicodeCharacter{00A0}{ } 

\usepackage{graphicx}
\usepackage{amsmath}
\usepackage{amssymb}
\usepackage{commath}
\usepackage{mathtools}
\usepackage{bbm}
\usepackage{verbatim}
\usepackage{color}
\usepackage[dvipsnames]{xcolor}
\usepackage{microtype}
\usepackage{url}

\input{plotting.tex}

\newcommand{\T}{\textsf{T}} 
\renewcommand{\b}[1]{\pmb{#1}} 
\newcommand{\textsm}[1]{\text{{\tiny #1}}} 

\DeclareMathOperator*{\argmin}{arg\,min}  

\DeclarePairedDelimiterX\Set[2]{\lbrace}{\rbrace}%
{ #1 \,:\, #2 } 

\newtheorem{assumption}[theorem]{Assumption}

\newcommand{\rev}[1]{#1}

\begin{document}

\title{On the Positivity and Magnitudes of Bayesian Quadrature Weights}

\author{Toni Karvonen \and Motonobu Kanagawa \and Simo S\"{a}rkk\"{a}}

\institute{T. Karvonen \at
              Department of Electrical Engineering and Automation \\
              Aalto University, Finland \\
              \email{toni.karvonen@aalto.fi}
            \and
            M. Kanagawa \at
            University of Tübingen and Max Planck Institute for Intelligent Systems, Germany \\
            \email{motonobu.kanagawa@gmail.com}            
           \and
           S. S\"{a}rkk\"{a} \at
           Department of Electrical Engineering and Automation \\
          Aalto University, Finland \\
          \email{simo.sarkka@aalto.fi}
}

\date{}

\maketitle

\begin{abstract}
This article reviews and studies the properties of Bayesian quadrature weights, which strongly affect stability and robustness of the quadrature rule.
Specifically, we investigate conditions that are needed to guarantee that the weights are positive or to bound their magnitudes.
First, it is shown that the weights are positive in the univariate case if the design points locally minimise the posterior integral variance and the covariance kernel is totally positive (e.g., Gaussian and Hardy kernels).
This suggests that gradient-based optimisation of design points may be effective in constructing stable and robust Bayesian quadrature rules.
Secondly, we show that magnitudes of the weights admit an upper bound in terms of the fill distance and separation radius if the RKHS of the kernel is a Sobolev space (e.g., Mat\'ern  kernels), suggesting that quasi-uniform points should be used.
A number of numerical examples demonstrate that significant generalisations and improvements appear to be possible, manifesting the need for further research.

\end{abstract}

\section{Introduction}

This article is concerned with \emph{Bayesian quadrature}~\cite{OHagan1991,RasmussenGhahramani2002,Briol2017}, a probabilistic approach to numerical integration and an example of a \emph{probabilistic numerical method} \cite{Larkin1972,Hennig2015,Cockayne2017}.
Let $\Omega$ be a subset of  $\mathbb{R}^d$, $d \geq 1$, and $\nu$ a Borel probability measure on $\Omega$. Given an \emph{integrand} $f \colon \Omega \to \mathbb{R}$, the task is to approximate the integral
\begin{equation*}
I_\nu(f) \coloneqq \int_\Omega f \dif \nu,
\end{equation*}
the solution of which is assumed not to be available in closed form.
In Bayesian quadrature, a user specifies a prior distribution over the integrand as a Gaussian process $f_\textsm{GP} \sim \mathcal{GP}(0,k)$ by choosing a positive-definite covariance kernel $k \colon \Omega \times \Omega \to \mathbb{R}$, so as to faithfully represent their knowledge about the integrand, such as its smoothness.
The user then evaluates the true integrand at chosen design points $X = \{\b{x}_1, \dots, \b{x}_n\} \subset \Omega$.
By regarding the pairs $\mathcal{D} \coloneqq \{(\b{x}_i, f(\b{x}_i))\}_{i=1}^n$ thus obtained as ``observed data'', the posterior distribution $I_\nu(f_\textsm{GP}) \mid \mathcal{D}$ becomes a Gaussian random variable.
This posterior distribution is useful for uncertainty quantification and decision making in subsequent tasks; this is one factor that makes Bayesian quadrature a promising approach in modern scientific computation, where quantification of discretisation errors is of great importance \cite{Briol2017,Oates2017}.

In Bayesian quadrature, the mean of the posterior over the integral is used as a quadrature estimate. The mean given is as a weighted average of function values:
\begin{equation*}
\mathbb{E}\big[ I_\nu(f_\textsm{GP}) \mid \mathcal{D} \big] = \sum_{i=1}^n w_{X,i}^\textsm{BQ} f(\b{x}_i) \approx \int_\Omega f \dif \nu,
\end{equation*}
where $w_{X,1}^\textsm{BQ}, \dots, w_{X,n}^\textsm{BQ} \in \mathbb{R}$ are the weights computed with the kernel $k$, design points $X$ and the measure $\nu$ (see Section \ref{sec:bq-basics} for details).
This form is similar to (quasi) Monte Carlo methods, where $\b{x}_1,\dots,\b{x}_n$ are (quasi) random points from a suitable proposal distribution and $w_1,\dots,w_n$ are the associated importance weights, positive by definition.
This similarity naturally leads to the following question: Are the weights $w_{X,1}^\textsm{BQ},\dots, w_{X,n}^\textsm{BQ}$ of Bayesian quadrature positive?
These weights are derived with no explicit positivity constraint, so in general some of them can be negative, \rev{which is observed in~\cite[Section 3.1.1]{HuszarDuvenaud2012}.}
Therefore, the question can be stated as: \emph{Under which conditions on the points and the kernel are the weights guaranteed to be positive?}

This question is important both conceptually and practically.
On the conceptual side, positive weights are more natural, given that the weighted sample $(w_i,\b{x}_i)_{i=1}^n$ can be interpreted as an approximation of the positive probability measure $\nu$; in fact, the Bayesian quadrature weights provide the best approximation of the representer of $\nu$ in the reproducing kernel Hilbert space (RKHS) of the covariance kernel, provided that $\b{x}_1,\dots,\b{x}_n$ are fixed (see Section \ref{sec:RKHS}).
Thus, if the weights are positive, then each weight $w_i$ can be interpreted as representing the ``importance'' of the associated point $\b{x}_i$ for approximating $\nu$.
This interpretation may be more acceptable to users familiar with Monte Carlo methods, encouraging them to adopt Bayesian quadrature.

On the practical side, quadrature rules with positive weights enjoy the advantage of being numerically more stable against errors in integrand evaluations.
In fact, besides Monte Carlo methods, many other practically successful or in some sense optimal rules have positive weights. Some important examples include Gaussian~\rev{\cite[Section~1.4.2]{Gautschi2004}} and
Clenshaw--Curtis quadrature~\rev{\cite{ClenshawCurtis1960}} and their tensor product extensions. \rev{Other domains besides subsets of $\mathbb{R}^d$ have also received their share of attention. For instance, positive-weight rules on the sphere are constructed in~\cite{Mhaskar2001} and interesting results connecting fill-distance and positivity of the weights of quadrature rules on compact Riemannian manifolds appear in~\cite{Breger2018}.}
It is also known that, in some typical function classes, such as Sobolev spaces, optimal rates of convergence can be achieved by considering only positive-weight quadrature rules; see for instance~\cite[Section 1]{Novak1999} and references therein. 
Therefore, if one can find conditions under which Bayesian quadrature weights are positive, then these conditions may be used as guidelines in construction of numerically stable Bayesian quadrature rules.

This article reviews existing, and derives new, results on properties of the Bayesian quadrature weights, focusing in particular on their \emph{positivity} and \emph{magnitude}.
One of our principal aims is to stimulate new research on quadrature weights in the context of probabilistic numerics.
While convergence rates of Bayesian quadrature rules have been studied extensively in recent years~\cite{Briol2017,Kanagawa2016,Kanagawa2018}, analysis of the weights themselves has not attracted much attention.
On the other hand, the earliest work~\cite{Larkin1970,RichterDyn1971a,BarrarLoeb1976} (see \cite{Oettershagen2017} for a recent review) done in the 1970s on kernel-based quadrature  already revealed certain interesting properties of the Bayesian quadrature weights. These results seem not well-known in the statistics and machine learning community.
Moreover, there are some useful results from the literature on scattered data approximation~\cite{DeMarchiSchaback2010}, which can be used to analyse the properties of Bayesian quadrature weights.
The basics of Bayesian quadrature are reviewed in Section \ref{sec:bq} while the main contents, including simulation results, of the article are presented in Sections \ref{sec:positivity} and \ref{sec:stability}.

In Section \ref{sec:positivity}, we present results concerning positivity of the Bayesian quadrature weights.
We discuss results on the number of the weights that must be positive, focusing on the univariate case and \emph{totally positive} kernels (Definition~\ref{def:totally-positive}).
Corollary~\ref{cor:optimal-weights}, the main result of this section, states that all the weights are positive if the design points are locally optimal.
A practically relevant consequence of this result is that it may imply that the weights are positive if the design points are obtained by gradient descent, which is guaranteed to provide locally optimal points (see e.g.~\cite{LeeSimJorRec16}).

Section \ref{sec:stability} focuses on results on the magnitudes of the weights.
More specifically, we discuss the behaviour of the sum of absolute weights, $\sum_{i=1}^n \abs[0]{w_{X,i}^\textsm{BQ}}$, that strongly affects stability and robustness of Bayesian quadrature.
If this quantity is small, the quadrature rule is robust against misspecification of the Gaussian process prior \cite{Kanagawa2018} and errors in integrand evaluations~\cite{Forster1993} \rev{and kernel means~\cite[pp.\ 298--300]{SommarivaVianello2006}}.
This quantity is also related to the numerical stability of the quadrature rule.
Using a result on stability of kernel interpolants by De Marchi and Schaback~\cite{DeMarchiSchaback2010}, we derive an upper bound on the the sum of absolute weights for some typical cases where the Gaussian process has finite degree of smoothness and the RKHS induced by the covariance kernel is norm-equivalent to a Sobolev space.

\section{Bayesian quadrature} \label{sec:bq}

This section defines a Bayesian quadrature rule as the integral of the posterior of Gaussian process used to model the integrand.
We also discuss the equivalent characterisation of this quadrature rule as the worst-case optimal integration rule in the RKHS $\mathcal{H}(k)$ induced by the covariance kernel $k$ of the Gaussian process.

\subsection{Basics of Bayesian quadrature} \label{sec:bq-basics}

In standard Bayesian quadrature~\cite{OHagan1991,Minka2000,Briol2017}, the deterministic integrand $f \colon \Omega \to \mathbb{R}$ is modelled as a Gaussian process. The integrand is assigned a zero-mean Gaussian process prior $f_\textsm{GP} \sim \mathcal{GP}(0, k)$ with a positive-definite covariance kernel $k$. This is to say that for any $n \in \mathbb{N}$ and any distinct points $X = \{\b{x}_1, \ldots, \b{x}_n \} \subset \Omega$ we have $(f_\textsm{GP}(\b{x}_1),\ldots,f_\textsm{GP}(\b{x}_n)) \sim \mathrm{N}(\b{0}, \b{K}_X)$, with
\begin{equation*}
    [\b{K}_X]_{ij} \coloneqq \mathrm{Cov}\big[f_\textsm{GP}(\b{x}_j), f_\textsm{GP}(\b{x}_i) \big] = k(\b{x}_j, \b{x}_i)
\end{equation*}
the $n \times n$ positive-definite (and hence invertible) \emph{kernel matrix}.
Conditioning on the \emph{data} $\mathcal{D} =\{(\b{x}_i,f(\b{x}_i))\}_{i=1}^n$, consisting of evaluations $\b{f}_{X} \coloneqq (f(\b{x}_i))_{i=1}^n \in \mathbb{R}^n$ of $f$ at points $X$, yields a Gaussian posterior process with the mean
\begin{equation} \label{eq:GP-posterior-mean}
\begin{split}
\mu_{X,f}(\b{x}) &\coloneqq \mathbb{E} \big[ f_\textsm{GP}(\b{x}) \mid \mathcal{D} \big] \\
&= \b{k}_X(\b{x})^\T \b{K}_X^{-1} \b{f}_X
\end{split}
\end{equation}
and covariance
\begin{equation*}
\begin{split}
\sigma^2_X(\b{x},\b{x}') &\coloneqq \mathrm{Cov}\big[f_\textsm{GP}(\b{x}), f_\textsm{GP}(\b{x}') \mid \mathcal{D} \big] \\
&= k(\b{x},\b{x}') - \b{k}_X(\b{x})^\T \b{K}_X^{-1} \b{k}_X(\b{x}'),
\end{split}
\end{equation*}
where the $n$-vector $\b{k}_X(\b{x})$ has the elements $[\b{k}_X(\b{x})]_i = k(\b{x},\b{x}_i)$.
Note that the posterior covariance only depends on the points, not on the integrand, and that the posterior mean \emph{interpolates} the data (i.e., $\mu_{X,f}(\b{x_i}) = f(\b{x}_i)$ for $i=1,\ldots,n$). Accordingly, the posterior mean often goes by the name \emph{kernel interpolant} or, if the kernel is isotropic, \emph{radial basis function interpolant}.

Due to the linearity of the integration operator, the posterior of the integral becomes a Gaussian distribution $I(f_\textsm{GP}) \mid \mathcal{D} \sim \mathcal{N}(I_X^\textsm{BQ}(f), \mathbb{V}_X^\textsm{BQ})$ with the mean and variance
\begin{align}
\begin{split}\nonumber
I_X^\textsm{BQ}(f) &\coloneqq \mathbb{E}\big[I_\nu(f_\textsm{GP}) \mid \mathcal{D}\big] \\
&= \int_\Omega \mathbb{E}\big[f_\textsm{GP}(\b{x}) \mid \mathcal{D}\big] \dif \nu(\b{x}) \\
&= \b{k}_{\nu,X}^\T \b{K}^{-1}_X \b{f}_X, 
\end{split} \\
\begin{split}\label{eq:bq-variance}
\mathbb{V}_X^\textsm{BQ} &\coloneqq \text{Var}[I(f_\textsm{GP}) \mid \mathcal{D}] \\
&= \int_\Omega \int_\Omega \mathrm{Cov}\big[f_\textsm{GP}(\b{x}), f_\textsm{GP}(\b{x}') \mid \mathcal{D} \big] \dif \nu(\b{x}) \dif \nu(\b{x}') \\
&= I_\nu(k_\nu) - \b{k}_{\nu,X}^\T \b{K}_X^{-1} \b{k}_{\nu,X}, 
\end{split}
\end{align}
where $k_\nu(\b{x}) \coloneqq \int_\Omega k(\cdot, \b{x}) \dif \nu(\b{x})$ is the \emph{kernel mean}~\cite{Smola2007}, $\b{k}_{\nu,X} \in \mathbb{R}^n$ with $[\b{k}_{\nu,X}]_i = k_\nu(\b{x}_i)$ and
\begin{equation*}
I_\nu(k_\nu) = \int_\Omega k_\nu(\b{x}) \dif \nu(x) = \int_\Omega \int_\Omega k(\b{x}, \b{x}') \dif \nu (\b{x}') \dif \nu(\b{x}).
\end{equation*}
The integral mean $I_X^\textsm{BQ}(f)$ is used to approximate the true intractable integral $I_\nu(f)$ while the variance $\mathbb{V}_X^\textsm{BQ}$ is supposed to quantify epistemic uncertainty, due to partial information being used (i.e., a finite number of function evaluations) inherent to this approximation.

The integral mean $I_X^\textsm{BQ}(f)$ indeed takes the form a quadrature rule, a weighted sum of function evaluations:
\begin{equation*}
    I^\textsm{BQ}_X(f) = (\b{w}_X^\textsm{BQ})^\T \b{f}_X = \sum_{i=1}^n w_{X,i}^\textsm{BQ} f(\b{x}_i),
\end{equation*}
where $w_{X,1}^\textsm{BQ}, \dots, w_{X,n}^\textsm{BQ}$  are the \emph{Bayesian quadrature weights} given by
\begin{equation}\label{eq:BQ-weights}
    \b{w}_X^\textsm{BQ} \coloneqq ( w_{X,i}^\textsm{BQ} )_{i=1}^n \coloneqq \b{K}_{X}^{-1} \b{k}_{\nu,X} \in \mathbb{R}^n.
\end{equation}
The purpose of this article is to analyse the properties of these weights.

A particular property of a Bayesian quadrature rule is that the $n$ kernel translates $k_{\b{x}_i} \coloneqq k(\cdot, \b{x}_i)$ are integrated exactly:
\begin{equation} \label{eq:BQ-exactness}
    I_X^\textsm{BQ}(k_{\b{x}_i}) = I_\nu(k_{\b{x}_i}) = k_\nu(\b{x}_i) \: \text{ for each } \: i = 1,\ldots,n,
\end{equation}
which is derived from the fact that the $j$th equation of the linear system $\b{K}_X \b{w}_X^\textsm{BQ} = \b{k}_{\nu,X}$ defining the weights is
\begin{equation*}
    \sum_{i=1}^n k(\b{x}_j, \b{x}_i) w_{X,i}^\textsm{BQ} = k_\nu(\b{x}_j).
\end{equation*}
The left-hand side is precisely $I_X^\textsm{BQ}(k_{\b{x}_j})$ while on the right-hand side we have $k_\nu(\b{x}_j) = I_\nu(k_{\b{x}_j})$.
Note also that the integral variance is the integration error of the kernel mean:
\begin{equation*}
\begin{split}
    \mathbb{V}_X^\textsm{BQ} &= I_\nu(k_\nu) - \b{k}_{\nu,X}^\T \b{K}_X^{-1} \b{k}_{\nu,X} \\
    &= I_\nu(k_\nu) - (\b{w}_X^\textsm{BQ})^\T \b{k}_{\nu,X} \\
    &= I_\nu(k_\nu) - I_X^\textsm{BQ}(k_\nu).
\end{split} 
\end{equation*}

Occasionally, it is instructive to interpret the weights as integrals of the \emph{Lagrange cardinal functions} \sloppy{${\b{u}_X = (u_{X,i})_{i=1}^n}$} (see e.g. \cite[Chapter 11]{Wendland2005}). 
These functions are defined as $\b{u}_X(\b{x}) = \b{K}_X^{-1} \b{k}_X(\b{x})$, from which it follows that
\begin{equation}\label{eq:lagrange-form}
\mu_{X,f}(\b{x}) = \b{u}_X(\b{x})^\T \b{f}_X = \sum_{i=1}^n f(\b{x}_i) u_{X,i}(\b{x}).
\end{equation}
Consequently, the cardinality property 
\begin{equation*}
u_{X,i}(\b{x}_j) = \delta_{ij} \coloneqq \begin{cases} 1 \: \text{ if } \: i=j, \\ 0 \: \text{ if } \: i \neq j. \end{cases}
\end{equation*}
is satisfied, as can be verified by considering the interpolant $\mu_{X,g_i}$ to any function $g_i$ such that $g_i(\b{x}_j) = \delta_{ij}$. Since the integral mean is merely the integral of $\mu_{X,f}$, we have from \eqref{eq:lagrange-form} that
\begin{equation*}
I_X^\textsm{BQ}(f) = \sum_{i=1}^n f(\b{x}_i) I_\nu(u_{X,i}).
\end{equation*}
That is, the $i$th Bayesian quadrature weight is the integral of the $i$th Lagrange cardinal function: $w_{X,i}^\textsm{BQ} = I_\nu(u_{X,i})$.

\subsection{Reproducing kernel Hilbert spaces} \label{sec:RKHS}

An alternative interpretation of Bayesian quadrature weights is that they are, for the given points, the worst-case optimal weights in the reproducing kernel Hilbert space $\mathcal{H}(k)$ induced by the covariance kernel $k$. The material of this section is contained in, for example,~\cite[Section 2]{Briol2017}, \cite[Section 3.2]{Oettershagen2017} and \cite[Section 2]{Karvonen2018}. For a comprehensive introduction to RKHSs, see the monograph of Berlinet and Thomas-Agnan~\cite{Berlinet2011}.

The RKHS induced by $k$ is the unique Hilbert space of functions characterised by (i) the \emph{reproducing property} $\langle k_{\b{x}}, f \rangle_{\mathcal{H}(k)} = f(\b{x})$ for every $f \in \mathcal{H}(k)$ and $\b{x} \in \Omega$ and (ii) the fact that $k_{\b{x}} \in \mathcal{H}(k)$ for every $\b{x} \in \Omega$. 
The \emph{worst-case error} in $\mathcal{H}(k)$ of a quadrature rule with points $X$ and weights $\b{w} \in \mathbb{R}^n$ is
\begin{equation*}
\begin{split}
e_{\mathcal{H}(k)}(X, \b{w})^2 &\coloneqq \sup_{\lVert f \rVert_{\mathcal{H}(k)} \leq 1} \Bigg\lvert \int_\Omega f \dif \nu - \sum_{i=1}^n w_i f(\b{x}_i) \Bigg\rvert \\
&= I_\nu(k_\nu) - 2 \b{w}^\T \b{k}_{\nu,X} + \b{w}^\T \b{K}_X \b{w}.
\end{split}
\end{equation*}
It can be then shown that the Bayesian quadrature weights $\b{w}_X^\textsm{BQ}$ are the unique minimiser of the worst-case error among all possible weights for these points: 
\begin{equation*}
\b{w}_X^\textsm{BQ} = \argmin_{\b{w} \in \mathbb{R}^n} e_{\mathcal{H}(k)}(X, \b{w})
\end{equation*}
and
\begin{equation} \label{eq:variance-is-wce}
\mathbb{V}_X^\textsm{BQ} = e_{\mathcal{H}(k)}(X, \b{w}_X^\textsm{BQ})^2.
\end{equation}
Furthermore, the worst-case error can be written as the RKHS error in approximating the integration representer $k_\nu$ that satisfies $I_\nu(f) = \langle k_\nu , f \rangle_{\mathcal{H}(k)}$ for all $f \in \mathcal{H}(k)$: 
\begin{equation*}
e_{\mathcal{H}(k)}(X, \b{w}_X^\textsm{BQ}) = \norm[0]{k_\nu - k_Q}_{\mathcal{H}(k)}, \quad k_Q \coloneqq \sum_{i=1}^n w_{X,i}^\textsm{BQ} k_{\b{x}_i}.
\end{equation*}
From this representation and the Cauchy--Schwarz inequality it follows that
\begin{equation*}
\begin{split}
\abs[0]{I_\nu(f) - I_X^\textsm{BQ}(f)} &= \abs[1]{ \langle k_\nu - k_Q , f \rangle_{\mathcal{H}(k)} } \\ 
&\leq \norm[0]{f}_{\mathcal{H}(k)} \norm[0]{k_\nu - k_Q}_{\mathcal{H}(k)} \\
&= \norm[0]{f}_{\mathcal{H}(k)} e_{\mathcal{H}(k)}(X, \b{w}_X^\textsm{BQ}).
\end{split}
\end{equation*}
For analysis of convergence of Bayesian quadrature rules as $n \to \infty$ it is therefore sufficient to analyse how the worst-case error (i.e., integral variance) behaves---as long as the integrand indeed lives in $\mathcal{H}(k)$. Convergence will be discussed in Section \ref{sec:stability}.


\section{Positivity} \label{sec:positivity}

This section reviews existing results on the positivity of the weights of Bayesian quadrature that can be derived in one dimension when the covariance kernel is \emph{totally positive}. This assumption, given in Definition \ref{def:totally-positive}, is stronger than positive-definiteness but is satisfied by, for example, the Gaussian kernel. For most of the section we assume that $d = 1$ and $\Omega = [a,b]$ for $a < b$. Furthermore, the measure $\nu$ is typically assumed to admit a density function with respect to the Lebesgue measure,\footnote{This can be usually relaxed to $I_\nu$ being a positive linear functional: $I_\nu(f) > 0$ whenever $f$ is almost everywhere positive.} an assumption that implies $I_\nu(f) > 0$ if $f(x) > 0$ for almost every $x \in \Omega$.

Positivity of the weights was actively investigated during the 1970s~\cite{Richter1970,RichterDyn1971a,RichterDyn1971b,BarrarLoebWerner1974,BarrarLoeb1976}, and these results have been recently refined and collected by Oettershagen~\cite[Section 4]{Oettershagen2017}.
To simplify presentation, some of the results in this section are given in a slightly less general form than possible.
Two of the most important results are
\begin{itemize}
\item \textbf{Theorem \ref{thm:BQ-positivity-general}}: At least one half of the weights of \emph{any} Bayesian quadrature rule are positive (Theorem \ref{thm:BQ-positivity-general}).
\item \textbf{Corollary \ref{cor:optimal-weights}}: All the weights are positive when the points are selected so that the integral posterior variance in~\eqref{eq:bq-variance} is locally minimised in the sense that each of its $n$ partial derivatives with respect to the integration points vanishes (Definition~\ref{def:bq-optimal}).
\end{itemize}
The latter of these results is particularly interesting since (i) it implies that points selected using a gradient descent algorithm may have positive weights and (ii) the resulting Bayesian quadrature rule is a positive linear functional and hence potentially well-suited for integration of functions that are known to be positive---a problem for which a number of transformation-based methods have been developed recently~\cite{Osborne2012,Gunter2014,Chai2018}.

As no multivariate extension of the theory used to prove the aforementioned results appears to have been developed, we do not provide any general theoretical results on the weights in higher dimensions. 
However, some special cases based on, for example, tensor products are discussed in Sections \ref{sec:other-positive} and \ref{sec:multi-positive} and two numerical examples are used to provide some evidence for the conjectures that multivariate versions of Theorem \ref{thm:BQ-positivity-general} and Corollary \ref{cor:optimal-weights} hold.

It will turn out that optimal Bayesian quadrature rules are analogous to classical Gaussian quadrature rules in the sense that, in addition to being exact for kernel interpolants (recall \eqref{eq:BQ-exactness}), they also exactly integrate Hermite interpolants (see Section \ref{sec:hermite-interpolation}). 
We thus begin by reviewing the argument used to establish positivity of the Gaussian quadrature weights. 

\subsection{Gaussian quadrature} \label{sec:gauss-quadrature}

Under the assumption that $\nu$ admits a density\footnote{This can be generalised to the cumulative distribution function having infinitely many points of increase.} there exist unique weights $w_1,\ldots,w_n$ and points $x_1,\ldots,x_n \in [a,b]$ such that
\begin{equation} \label{eq:Gaussian-quadrature}
    \sum_{i=1}^n w_i P(x_i) = \int_a^b P(x) \dif \nu(x)
\end{equation}
for every polynomial $P$ of degree at most $2n-1$~\cite[Chapter 1]{Gautschi2004}. This quadrature rule is known as a \emph{Gaussian quadrature rule} (for the measure $\nu$). 
One can show the positivity of the weights of a Gaussian rule as follows.

\begin{proposition}
Assume that $\nu$ admits a Lebesgue density.
Then the weights $w_1,\dots,w_n$ of the Gaussian quadrature \eqref{eq:Gaussian-quadrature} are positive.
\end{proposition}
\begin{proof}
For each $i=1,\ldots,n$ there exists a unique polynomial $L_i$ of degree $n-1$ such that $L_i(x_j) = \delta_{ij}$. 
This property is shared by the function $G_i \coloneqq L_i^2 \geq 0$ that, being of degree $2n-2$, is also integrated exactly by the Gaussian rule. Because $G_i$ is almost everywhere positive, it follows from the assumption that $\nu$ admits a density that
\begin{equation*}
    0 < \int_a^b G_i(x) \dif \nu(x) = \sum_{i=1}^n w_j G_i(x_j) = w_i.
\end{equation*}
The positivity of the weights is thus concluded. \qed
\end{proof}
This proof may appear to be based on the closedness of the set of polynomials under exponentiation. 
Closer analysis reveals a structure that can be later generalised.

To describe this, recall that one of the basic properties of polynomials is that a polynomial $P$ of degree $n$ can have at most $n$ zeroes, when counting multiplicities (for some properties of polynomials and interpolation with them, see e.g. \cite[Chapter 3]{Atkinson1989}). 
This is to say that, if for some points $x_1,\ldots,x_m$ it holds that
\begin{equation*}
    P^{(j_i)}(x_i) \coloneqq \frac{\mathrm{d}^{j_i}}{\dif x^{j_i}} P(x) \Bigr|_{x = x_i} = 0
\end{equation*}
for $j_i = 0,\ldots,q_i-1$, with $q_i$ being the \emph{multiplicity} of the zero $x_i$ of $P$, then $\sum_{i=1}^m q_i \leq n$. 
This fact on zeroes of polynomials can be used to supply a proof of positivity of the Gaussian quadrature weights that does not explicitly use of the fact that square of a function is non-negative.
By the chain rule, the derivative of $G_i$ vanishes at each $x_j$ such that $j \neq i$.
That is, $G_i$ has a double zero at each of these $n-1$ points (i.e., $G_i(x_j) = 0$ and $G_i^{(1)}(x_j) = 0$), for the total of $2n-2$ zeroes.
Being a polynomial of degree $2n-2$, $G_i$ cannot have any other zeroes besides these. 
Since all the zeroes of $G_i$ are double, it cannot hence have any sign changes.
This is because, in general, a function $g$ that satisfies $g(x) = g^{(1)}(x) = 0$ but $g^{(2)}(x) \neq 0$ at a point $x$ cannot change its sign at $x$, since its derivative changes sign at the point. 
From $G_i(x_i) = 1 > 0$ it then follows that $G_i$ is almost everywhere positive.

\subsection{Chebyshev systems and generalised Gaussian quadrature} \label{sec:gen-gauss-quadrature}

The argument presented above works almost as such when the polynomials are replaced with generalised polynomials and the Gaussian quadrature rule with a generalised Gaussian quadrature rule. 
Much of the following material is covered by the introductory chapters of the monograph by Karlin and Studden~\cite{KarlinStudden1966}.
In the following $C^m([a,b])$ stands for the set of functions that are $m$ times continuously differentiable on the open interval $(a,b)$.

\begin{definition}[Chebyshev system] A collection of functions $\{\phi_i\}_{i=1}^m \subset C^{m-1}([a,b])$ constitutes an (extended) \emph{Chebyshev system} if any non-trivial linear combination of the functions, called a \emph{generalised polynomial}, has at most $m-1$ zeroes, counting multiplicities.
\end{definition}

\begin{remark} \label{remark:chebyshev} Some of the results we later present, such as Proposition \ref{prop:weight-positivity}, are valid even when a less restrictive definition, that does not require differentiability of $\phi_i$, of a Chebyshev system is used. Of course, in this case the definition is not given in terms of multiple zeroes. The above definition is used here to simplify presentation.
\rev{The simplest relaxation is to require that $\{\phi_i\}_{i=1}^m$ are merely continuous and that no linear combination can vanish at more than $m-1$ points.}
\end{remark}

By selecting $\phi_i(x) = x^{i-1}$ we see that polynomials are an example of a Chebyshev system. 
Perhaps the simplest example of a non-trivial Chebyshev system is given by the following example.

\begin{example}
Let $\phi_i(x) = \mathrm{e}^x x^{i-1}$ for $i=1,\dots,m$.
Then
$\{ \phi_i \}_{i=1}^m$ constitute a Chebyshev system.
To verify this, observe that any linear combination $\phi$ of $\phi_1,\ldots,\phi_m$ is of the form $\phi(x) = \mathrm{e}^x P(x)$ for a polynomial $P$ of degree at most $m-1$ and that the $j$th derivative of this function takes the form
\begin{equation} \label{eq:phi-derivative}
    \phi^{(j)}(x) = \mathrm{e}^x \big[ P(x) + c_1 P^{(1)}(x) + \cdots + c_j P^{(j)}(x) \big]
\end{equation}
for certain integer coefficients $c_1,\ldots,c_j$. 
We observe that $\phi(x_0) = 0$ for a point $x_0$ if and only if $P(x_0) = 0$. If also $\phi^{(1)}(x_0) = 0$, then it follows from \eqref{eq:phi-derivative} that $P^{(1)}(x_0) = 0$, and, generally, that $\phi^{(i)}(x_0) = 0$ for $i=0,\ldots,j$ if and only if $P^{(i)}(x_0) = 0$. That is, the zeroes of $\phi$ are precisely those of $P$ and, consequently, the functions $\phi_i$ constitute a Chebyshev system.
\end{example}

\subsubsection{Interpolation using a Chebyshev system}

A crucial property of generalised polynomials is that unique interpolants can be constructed using them, as we next show.
For any Chebyshev system $\{\phi_i\}_{i=1}^n$ and a set of distinct points $X = \{x_1,\ldots,x_n\} \subset [a,b]$, we know that there cannot exist $\b{\alpha} = (\alpha_1,\ldots,\alpha_n) \neq \b{0}$ such that
\begin{equation*}
    \sum_{i=1}^n \alpha_i \phi_i(x_j) = 0 \quad \text{ for every } \quad j = 1,\ldots,n
\end{equation*}
since $\alpha_1 \phi_1 + \cdots + \alpha_n \phi_n$ can have at most $n-1$ zeroes. Equivalently, the only solution $\b{\beta} \in \mathbb{R}^n$ to the linear system $\b{V}_X^\T \b{\beta} = \b{0}$ defined by the $n \times n$ matrix $[\b{V}_X]_{ij} = \phi_i(x_j)$ is $\b{\beta} = \b{0}$. That is, $\b{V}_X$ is invertible.

For any data $\{(x_i,f(x_i))\}_{i=1}^n$, the above fact guarantees the existence and uniqueness of an interpolant $s_{X,f}$ such that (i) $s_{X,f}$ is in $\mathrm{span}\{\phi_1,\ldots,\phi_n\}$ and (ii) $s_{X,f}(x_j) = f(x_j)$ for each $j=1,\ldots,n$.
These two requirements imply that
\begin{equation*}
    s_{X,f}(x_j) = \sum_{i=1}^n \alpha_i \phi_i(x_j) = f(x_j)
\end{equation*}
for and some $\b{\alpha} \in \mathbb{R}^n$ and every $j = 1,\dots, n$. 
In matrix form, these $n$ equations are equivalent to $\b{V}_X^\T \b{\alpha} = \b{f}_X$. Hence $\b{\alpha} = \b{V}_X^{-\T} \b{f}_X$ and the interpolant is
\begin{equation} \label{eq:interpolant-matrix-form}
s_{X,f}(x) = \b{\phi}(x)^\T \b{\alpha} = \b{\phi}(x)^\T \b{V}_X^{-\T} \b{f}_X
\end{equation}
for $[\b{\phi}(x)]_i = \phi_i(x)$ an $n$-vector. 

\subsubsection{Hermite interpolants} \label{sec:hermite-interpolation}

A \emph{Hermite interpolant} $s_{X,\b{q},f}$ is based on data containing also derivative values (see \cite[Section 3.6]{Atkinson1989} for polynomial and \cite[Chapter 36]{Fasshauer2007} for kernel-based Hermite interpolation). In this setting, the point set $X$ contains $m$ points and $\b{q} \in \mathbb{N}_0^m$ is a vector of multiplicities such that $\sum_{i=1}^m q_i = n$. The data to be interpolated is
\begin{equation*}
\Set{(x_i, f^{(j_i)}(x_i))}{i = 1,\ldots,m \text{ and } j_i = 0, \ldots , q_i - 1}.
\end{equation*}
That is, the interpolant is to satisfy
\begin{equation*}
        s_{X,\b{q},f}^{(j_i)}(x_i) = f^{(j_i)}(x_i)
\end{equation*}
for each $i=1,\ldots,m$ and $j_i = 0, \ldots, q_i - 1$. 
Note that the interpolant $s_{X,f}$ is a Hermite interpolant with $m = n$ and $q_1 = \cdots = q_n = 1$. If the interpolant is to lie in $\mathrm{span}\{\phi_1,\ldots,\phi_n\}$, we must have, for some $\alpha_1,\ldots,\alpha_n$,
\begin{equation*}
s_{X,\b{q},f}^{(j_i)}(x_i) = \sum_{l=1}^n \alpha_l \phi_l^{(j_i)}(x_i) = f^{(j_i)}(x_i).
\end{equation*}
Again, these $n$ equations define a linear system that is invertible because $\{\phi_i\}_{i=1}^n$ constitute a Chebyshev system. The Hermite interpolant can be written in the form \eqref{eq:interpolant-matrix-form} with $\b{V}_X$ replaced with a version involving also derivatives of $\phi_i$ (see e.g. \cite[Section 2.3.1]{Oettershagen2017}).

\begin{figure}[t!]
\centering
  \includegraphics{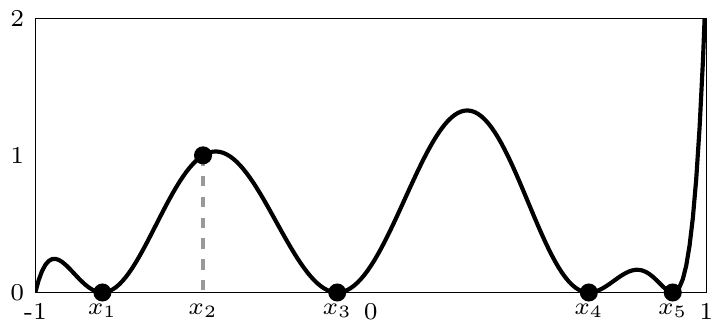}
  \caption{Example of a Hermite interpolant $F_i$ used in proving positivity of the weights of generalised Gaussian quadrature rule. This figure uses the Chebyshev system formed by $\phi_i(x) = \mathrm{e}^x x^{i-1}$.}\label{fig:hermite-interp}
\end{figure}

\subsubsection{Generalised Gaussian quadrature}

A \emph{generalised Gaussian quadrature rule} is a quadrature rule that uses $n$ points to integrate exactly all functions in the span of $\{\phi_i\}_{i=1}^{2n}$ constituting a Chebyshev system:
\begin{equation} \label{eq:gen-Gaussian-quadrature}
    \sum_{i=1}^n w_i \phi(x_i) = \int_a^b \phi(x) \dif \nu(x)
\end{equation}
for every $\phi \in \mathrm{span}\{\phi_1,\ldots,\phi_{2n}\}$.
The existence and uniqueness of the points and weights is guaranteed under fairly general assumptions~\cite{Barrow1978}. 
We prove positivity of the weights by constructing a function \sloppy{${F_i \in \mathrm{span}\{\phi_1,\ldots,\phi_{2n}\}}$} analogous to $G_i$ in Section \ref{sec:gauss-quadrature}. 

\begin{proposition}
Assume that $\nu$ admits a Lebesgue density.
Then the weights $w_1,\dots,w_n$ of the generalised Gaussian quadrature rule \eqref{eq:gen-Gaussian-quadrature} are positive.
\end{proposition}
\begin{proof}
Let $F_i$ be the Hermite interpolant to the data
\begin{equation*}
    f(a) = 0, \quad f(x_i) = 1, \quad f(x_j) = f^{(1)}(x_j) = 0 \text{ for } j \neq i.
\end{equation*}
An example is depicted in Figure~\ref{fig:hermite-interp}.
As there are $2n$ data points, $F_i$ indeed exists since $\{\phi_i\}_{i=1}^{2n}$ are a Chebyshev system. 
Moreover, $F_i$ has $2n-1$ zeroes. 
Because all its zeroes occurring on $(a,b)$ are double, $F_i$ cannot have sign-changes. Since $F_i(x_i) = 1 > 0$, we conclude $F_i$ is almost everywhere positive. Consequently, $w_i = I_\nu(F_i) > 0$.  \qed
\end{proof}

Next we turn our attention to kernels whose translates and their derivatives constitute Chebyshev systems.

\subsection{Totally positive kernels} \label{sec:totally-positive}

We are now ready to begin considering kernels and Bayesian quadrature. A concept related to Chebyshev systems is that of totally positive kernels whose theory is covered by the monograph of Karlin \cite{Karlin1968}. For a sufficiently differentiable kernel, define the derivatives
\begin{equation}\label{eq:kernel-derivative}
    k^{(j)}_y(x) \coloneqq k^{(j)}(x,y) \coloneqq \frac{\partial^j}{\partial z^j} k(x,z) \Bigr|_{z = y}.
\end{equation}
If the derivative
\begin{equation*}
  \frac{\partial^{2j}}{\partial x^j \partial y^j} k(x,y)
\end{equation*}
exists and is continuous for every $j \leq m$, the kernel is said to be $m$ times continuously differentiable, which we denote by writing $k \in C^{m}([a,b]^2)$. In this case, $f \in C^m([a,b])$ if $f \in \mathcal{H}(k)$ and that the kernel derivatives~\eqref{eq:kernel-derivative} act as representers for differentiation (i.e., $\langle f, k^{(j)}(\cdot,x) \rangle_{\mathcal{H}(k)} = f^{(j)}(x)$ for $f \in \mathcal{H}(k)$ and $j \leq m$); see \cite[Corollary 4.36]{Steinwart2008} and its proof. 

\begin{definition}[Totally positive kernel]
\label{def:totally-positive} 
A kernel $k \in C^\infty([a,b]^2)$ is (extended) \emph{totally positive of order $q \in \mathbb{N}$} if the collection
\begin{equation*}
    \big\{ k_{x_i}^{(j_i)} \, \colon \, i = 1,\ldots,m \text{ and } j_i = 0,\ldots,q_i-1 \big\}
\end{equation*}
constitutes a Chebyshev system for any $m \in \mathbb{N}$, any distinct $x_1,\ldots,x_m \in \Omega$ and any multiplicities $q_1,\ldots,q_m \leq q$ of these points.
\end{definition}

The class of totally positive kernels is smaller than that of positive-definite kernels. For the simplest case of $q=1$ and $m = n$ the total positivity condition is that the kernel translates $k_{x_1},\ldots,k_{x_n}$ constitute a Chebyshev system. This implies that the $n \times n$ matrix $[\b{K}_{Y,X}] \coloneqq k(y_j,x_i)$, which is just the matrix $\b{V}_Y$ considered in Section~\ref{sec:gen-gauss-quadrature} for the Chebyshev system $\phi_i = k_{x_i}$, is invertible for any \sloppy{${Y = \{y_1,\ldots,y_n\} \subset [a,b]}$}. Positive-definiteness of $k$ only guarantees that $\b{K}_{Y,X}$ is invertible when $Y = X$. 

Basic examples of totally positive kernels are the Gaussian kernel 
\begin{equation} \label{eq:gauss-kernel}
    k(x,x') = \exp\bigg( - \frac{(x-x')^2}{2\ell^2} \bigg) 
\end{equation}
with length-scale $\ell > 0$ and the Hardy kernel \sloppy{${k(x,x') = r^2/(r^2-xx')}$} for $r > 0$. Both of these kernels are totally positive of any order. 
There is also a convenient result that guarantees total positivity \cite[Proposition 3]{Burbea1976}: $k$ is totally positive if there are positive constants $a_m$ and a positive increasing function $v \in C^\infty([a,b])$ such that
\begin{equation*}
k(x,x') = \sum_{m=0}^\infty a_m v(x)^m v(x')^m
\end{equation*}
for all $x,x' \in \Omega$. More examples are collected in \cite{Karlin1968,Burbea1976}.


\subsection{General result on weights}

The following special case of the theory developed in \cite[Chapter 2]{KarlinStudden1966} appears in, for instance, \cite[Lemma 2]{RichterDyn1971a}. Its proof is a generalisation of the proof for the case $m=2n$ that was discussed in Section \ref{sec:gen-gauss-quadrature}.

\begin{proposition}\label{prop:weight-positivity} Suppose that $\{\phi_i\}_{i=1}^m \subset C^{m-1}([a,b])$ constitute a Chebyshev system, that $\nu$ admits  a Lebesgue density 
and that $Q(f) \coloneqq \sum_{i=1}^n w_i f(x_i)$ for $x_1,\ldots,x_m \in \Omega$ is a quadrature rule such that \sloppy{${Q(\phi_i) = I_\nu(\phi_i)}$} for each $i = 1,\ldots,m$. 
Then at least $\lfloor (m+1)/2 \rfloor$ of the weights $w_1, \ldots, w_n$ are positive.
\end{proposition}

An immediate consequence of this proposition is that a Bayesian quadrature rule based on a totally positive kernel has at least one half of its weights positive.

\begin{theorem}\label{thm:BQ-positivity-general} Suppose that the kernel $k \in C^\infty([a,b]^2)$ is totally positive of order 1. Then, for any points, at least $\lfloor (n+1)/2 \rfloor$ of the Bayesian quadrature weights $w_{X,1}^\textsm{BQ}, \ldots, w_{X,n}^\textsm{BQ}$ are positive.
\end{theorem}
\begin{proof} Since the kernel is totally positive of order 1, the translates $\{k_{x_i}\}_{i=1}^n$ constitute a Chebyshev system. The exactness condition \eqref{eq:BQ-exactness} holds for each of these functions. The claim follows by setting $m = n$ in Proposition~\ref{prop:weight-positivity}.
\qed
\end{proof}


\subsection{Weights for locally optimal points} \label{sec:positivity-optimal}

Recall the definition of the Bayesian quadrature variance:
\begin{equation*}
    \mathbb{V}_X^\textsm{BQ} = I_\nu(k_\nu) - \sum_{i=1}^n w_{X,i}^\textsm{BQ} k_\nu(x_i) = I_\nu(k_\nu) - \b{k}_{\nu,X}^\T \b{K}_X^{-1} \b{k}_{\nu,X}.
\end{equation*}
The variance can be considered a function $X \mapsto \mathbb{V}_X^\textsm{BQ}$ defined on the simplex
\begin{equation*}
    \mathcal{S}^n \coloneqq \{ \b{z} \in [a,b]^n \, : \, a < z_1 < \cdots < z_n < b  \} \subset [a,b]^n.
\end{equation*}
We introduce the following definition of locally optimal points. \rev{For this purpose, define the function
\begin{equation*}
    E(Z) \coloneqq \mathbb{V}_Z^\textsm{BQ} \quad \text{ for } \quad Z = (z_1, \ldots, z_n) \in \mathcal{S}^n
\end{equation*}
and its partial derivatives
\begin{equation*}
    E_j(X) \coloneqq \frac{\partial}{\partial z_j} E(Z) \Bigr|_{Z = X}.
\end{equation*}
}

\begin{definition} \label{def:bq-optimal} Let $m \leq n$. A Bayesian quadrature rule with points $X \subset [a,b]$ is \emph{locally $m$-optimal} if $X \in \mathcal{S}^n$ and \rev{there is an index set $\mathcal{I}_m^* \subset \{1, \ldots, n\}$ of $m$ indices such that
\begin{equation} \label{eq:optimality-condition}
    E_j(X) = \frac{\partial}{\partial z_j} \mathbb{V}_Z^\textsm{BQ} \Bigr|_{Z = X} = 0 \; \text{ for every } \; j \in \mathcal{I}_m^*.
\end{equation}
}
A locally $n$-optimal rule is called \emph{locally optimal}. 
The point set of a locally $m$-optimal Bayesian quadrature rule is also called locally $m$-optimal.
\end{definition}

When the kernel is totally positive of any order, it has been shown that any local minimiser of $\mathbb{V}_X^\textsm{BQ}$ is locally optimal in the sense of above definition. That is, no point in a point set that locally minimises the variance can be located on the boundary of the integration interval nor can any two points in the set coalesce.\footnote{Coalescence is possible because $\mathbb{V}_X^\textsm{BQ}$ is in fact a continuous function of $X$ defined on the whole of $\Omega^n$, not merely on $\mathcal{S}^n$~\cite[Proposition 5.5]{Oettershagen2017}. Coalescence of some of the points would result in a quadrature rule that uses also evaluations of derivatives of the integrand.}
These results, the origins of which can be traced to the 1970s~\cite{BarrarLoebWerner1974,BarrarLoeb1976,Bojanov1979}, have been recently collated by Oettershagen~\cite[Corollary 5.13]{Oettershagen2017}.

A locally $m$-optimal Bayesian quadrature rule is, in addition to the kernel translates at $X$, exact for translate derivatives at $x_j$ with $j \in \mathcal{I}_m^*$ (it is worth noting that Bayesian quadrature rules with derivative evaluations have been recently considered in \cite{PruherSarkka2016,WuAoiPillow2018}). When $m = n$, this is analogous to the interpretation of classical Gaussian quadrature rules as integrated Hermite interpolants~\cite{RichterDyn1971b}. This result first appeared in \cite{Larkin1970}. Its proof is typically based on considering the RKHS representation
\begin{equation*}
    \mathbb{V}_X^\textsm{BQ} = \norm[3]{k_\nu - \sum_{i=1}^n w_{X,i}^\textsm{BQ} k_{x_i}}_{\mathcal{H}(k)}^2
\end{equation*}
of the variance; see \cite[Section 3]{RichterDyn1971a} or \cite[Section 5.1.3]{Oettershagen2017}.
We present a mainly linear algebraic proof.

\begin{proposition} \label{prop:derivative-integration} Let  $m \leq n$. Suppose that the $n$-point set $X \in \mathcal{S}^n$ is locally $m$-optimal. If the kernel $k$ is once continuously differentiable, then
\begin{equation} \label{eq:m-optimal-exactness}
    \begin{split}
    I_X^\textsm{BQ}(k_{x}) &= I_\nu(k_{x}) \quad &&\text{for } \quad x \in X, \\
    I_X^\textsm{BQ}(k^{(1)}_{x_j}) &= I_\nu(k^{(1)}_{x_j}) \rev{\: \text{ or } \: w_{X,j}^\textsm{BQ} = 0} \quad &&\text{for } \quad j \in \mathcal{I}_m^*,
\end{split}
\end{equation}
where $k_x^{(1)}$ is the kernel derivative defined in~\eqref{eq:kernel-derivative}.
\end{proposition}

\begin{proof} By definition of local $m$-optimality, the partial derivatives
\begin{equation*}
    E_j(X) = \frac{\partial}{\partial z_j} E(Z) \Bigr|_{Z = X}
\end{equation*}
must vanish for each $j \in \mathcal{I}_m^*$. Let $\partial_i \b{g}(X) \in \mathbb{R}^n$ stand for the $i$th partial derivative of a vector-valued function $\b{g} \colon \mathbb{R}^n \to \mathbb{R}^n$ evaluated at $X$. From the explicit expression~\eqref{eq:bq-variance} for the variance we compute
\begin{equation*}
\begin{split}
    E_j(X) ={}& -2(\partial_j \b{k}_{\nu,X}^\T) \b{K}_X^{-1} \b{k}_{\nu,X} \\
    &+ \b{k}_{\nu,X}^\T \b{K}_X^{-1} (\partial_j \b{K}_X) \b{K}_X^{-1} \b{k}_{\nu,X} \\
    ={}& -2(\partial_j \b{k}_{\nu,X}^\T) \b{w}_X^\textsm{BQ} + (\b{w}_X^\textsm{BQ})^\T (\partial_j \b{K}_X) \b{w}_X^\textsm{BQ},
\end{split}
\end{equation*}
where the inverse matrix derivative formula
\begin{equation*}
\frac{\dif}{\dif x} \b{A}(x)^{-1} = -\b{A}(x)^{-1} \bigg[ \frac{\dif}{\dif x} \b{A}(x) \bigg] \b{A}(x)^{-1}
\end{equation*}
and the weight expression $\b{w}_X^\textsm{BQ} = \b{K}_X^{-1} \b{k}_{\nu,X}$ have been used.
The two partial derivatives appearing in the equation for $E_j(X)$ can be explicitly computed. First, only the $j$th element of $\b{k}_{\nu,X}$ depends on $x_j$. Thus,
\begin{equation*}
    [\partial_j \b{k}_{\nu,X}]_i = \frac{\partial}{\partial x_j} \int_a^b k(x,x_i) \dif \nu(x) = I_\nu(k_{x_j}^{(1)}) \delta_{ij}.
\end{equation*}
Secondly, only the $j$th row and column of $\b{K}_X$ have dependency on $x_j$. For $l \neq j$ we have
\begin{equation*}
    [\partial_j \b{K}_X]_{lj} = [\partial_j \b{K}_X]_{jl} = \frac{\partial}{\partial z} k(x_l,z) \Bigr|_{z = x_j} = k_{x_j}^{(1)}(x_l),
\end{equation*}
where the first equality is consequence of symmetry of the kernel. The diagonal element is a total derivative:
\begin{equation*}
\begin{split}
    [\partial_j \b{K}_X]_{jj} = \frac{\dif}{\dif x_j} k(x_j,x_j) &= 2\frac{\partial}{\partial z} k(x_j,z) \Bigr|_{z = x_j} \\
    &= 2k_{x_j}^{(1)}(x_j).
\end{split}
\end{equation*}
Therefore $\partial_j \b{K}_X$ is a zero matrix except for the $j$th row and column that are
\begin{equation*}
    \begin{bmatrix} k^{(1)}_{x_j}(x_1) & \cdots & k^{(1)}_{x_{j}}(x_{j-1}) & 2k^{(1)}_{x_j}(x_j) & k^{(1)}_{x_{j}}(x_{j+1}) & \cdots & k^{(1)}_{x_j}(x_n)\end{bmatrix}
\end{equation*}
and its transpose, respectively.
Hence
\begin{equation*}
\begin{split}
    E_j(X) &= -2 w_{X,j}^\textsm{BQ} I_\nu(k_{x_j}^{(1)}) + \sum_{i=1}^n \sum_{l=1}^n w_{X,i}^\textsm{BQ} w_{X,l}^\textsm{BQ} [\partial_j \b{K}_X]_{il} \\
    &= -2 w_{X,j}^\textsm{BQ} I_\nu(k_{x_j}^{(1)}) + 2 w_{X,j}^\textsm{BQ} \sum_{i=1}^n w_{X,i}^\textsm{BQ} k_{x_j}^{(1)}(x_i) \\
    &= - 2 w_{X,j}^\textsm{BQ}\big[I_\nu(k_{x_j}^{(1)}) - I_X^\textsm{BQ}(k_{x_j}^{(1)})\big].
\end{split}
\end{equation*}
If $w_{X,j}^\textsm{BQ} \neq 0$, then $E_j(X) = 0$ so that the form of $E_j$ above implies that $I_X^\textsm{BQ}(k_{x_j}^{(1)}) = I_\nu(k_{x_j}^{(1)})$. This concludes the proof. \qed
\end{proof}

\begin{remark}
Proposition \ref{prop:derivative-integration} admits an obvious multivariate extension~\cite[\foreignlanguage{russian}{Теорема} 2]{Gavrilov1998}: when $d > 1$, the $md$ partial derivative representers
\begin{equation*}
    \frac{\partial}{\partial z_j} k(\cdot, \b{z}) \Bigr|_{\b{z} = \b{x}_i}
\end{equation*}
for $j = 1,\ldots,d$ and $i \in \mathcal{I}_m^*$ are integrated exactly by a locally $m$-optimal Bayesian quadrature rule, defined by requiring a gradient version of \eqref{eq:optimality-condition}. See also~\cite{Gavrilov2007}. However, there appear to exist no generalisations of Chebyshev systems and Proposition~\ref{prop:weight-positivity} to higher dimensions.
\end{remark}

\rev{
\begin{theorem} \label{thm:m-optimal-weights}
Let $k \in C^\infty([a,b]^2)$ be a totally positive kernel of order 2 and $m \leq n$. Suppose that the point set $X \in \mathcal{S}^n$ is locally $m$-optimal with an index set $\mathcal{I}_m^* \subset \{1, \ldots, n\}$ and that the weights associated with $q \leq m$ indices in $\mathcal{I}_m^*$ are non-zero. Then at least $\lfloor (n+2m-q+1)/2 \rfloor$ of the weights are non-negative, and $q$ must satisfy $2m-n \leq q$.
\end{theorem}
\begin{proof}
By \eqref{eq:m-optimal-exactness}, the Bayesian quadrature rule in the statement is exact for $n$ kernel translates and $q$ of their derivatives.
By the total positivity of the kernel, the collection of these $n+q$ functions constitutes a Chebyshev system. 
By Proposition~\ref{prop:weight-positivity}, at least $\lfloor (n+q+1)/2 \rfloor$ of the weights are positive.
Since the weights associated with $m-q$ indices in $\mathcal{I}_m^*$ are zero, it follows that at least $\lfloor (n+q+1)/2 \rfloor + m - q = \lfloor (n+2m-q+1)/2 \rfloor$ of the weights are non-negative. 
The lower-bound for $q$ follows because $ \lfloor (n+2m-q+1)/2 \rfloor \leq n$ implies that $n+2m-q+1 \leq 2n+1$.
\qed
\end{proof}
}

The main result of this section follows by setting $m=n$ in the preceding theorem and observing that this implies $q=n$, which means that there can be no zero weights.

\begin{corollary} \label{cor:optimal-weights} 
If $k \in C^\infty([a,b]^2)$ is totally positive of order 2 and $X \in \mathcal{S}^n$ is locally optimal, then all the Bayesian quadrature weights $w_{X,1}^\textsm{BQ},\ldots,w_{X,n}^\textsm{BQ}$ are positive.
\end{corollary}

\begin{remark}
A key consequence of Corollary \ref{cor:optimal-weights} is the following: If $w_{X,1}^\textsm{BQ},\ldots,w_{X,n}^\textsm{BQ}$ contain negative values, then the design points $X$ are {\em not} locally optimal.
In other words, in this case there is still room for improvement by optimising these points using, for example, gradient descent.
In this way, the signs of the weights can provide information about the quality of the design point set.
\end{remark}

A positive-weight quadrature rule is a positive linear functional (i.e., every positive function is mapped to a positive real). 
A locally optimal Bayesian quadrature rule may therefore be appropriate for numerical integration of functions that are a priori known to be positive, such as likelihood functions. 
Theoretical comparison to warped models~\cite{Osborne2012,Gunter2014,Chai2018} that encode positivity of the integrand by placing the GP prior on, for example, square root of the integrand would be an interesting topic of research. 

\begin{figure*}[t!]
\centering
  \includegraphics{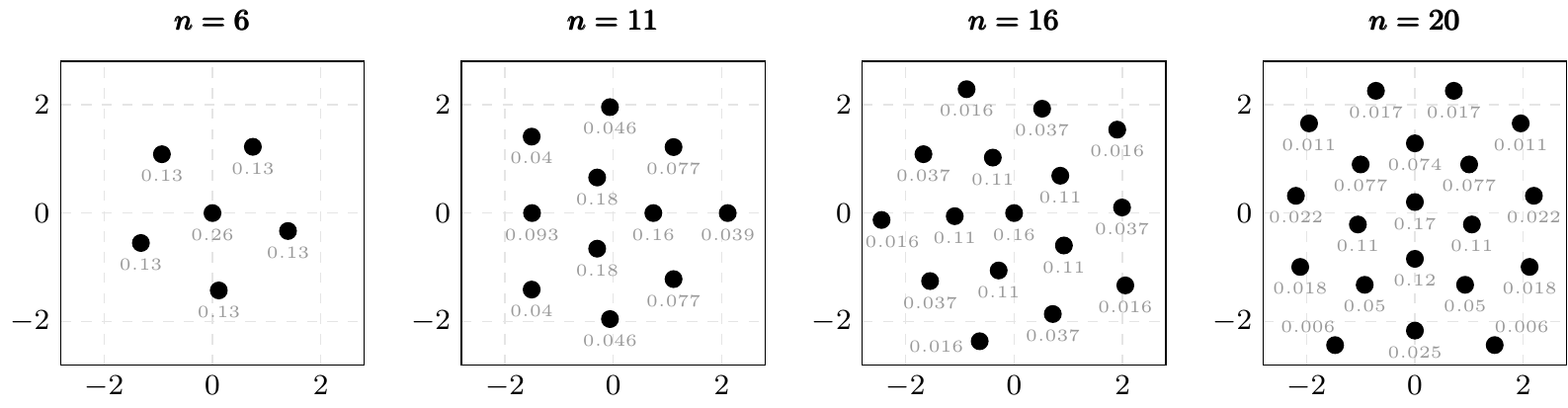}
  \caption{Locally optimal Bayesian quadrature point sets for the Gaussian measure and kernel on $\mathbb{R}^2$. The corresponding weights are written in grey. The sums of weights are $0.91~(n=6)$, $0.978~(n=11)$, $0.9975~(n=16)$ and $1.011~(n=20)$.
  }\label{fig:optimal-2d}
\end{figure*}

\subsection{Greedily selected points}

The optimal points discussed in the preceding section cannot be constructed efficiently (see \cite[Section 5.2]{Oettershagen2017} for what appears to be the most advanced published algorithm). In practice, points selected by greedy minimisation of the integral variance are often used. This approach is known as \emph{sequential Bayesian quadrature}~\cite{CookClayton1998,HuszarDuvenaud2012}. Assuming for a moment that $d$ is arbitrary and an $n$-point set $X_n \subset \Omega$ has been already generated, sequential Bayesian quadrature proceeds by selecting a new point $\b{x}_{n+1} \in \Omega$ by minimising the integral variance:
\begin{equation*}
\b{x}_{n+1} = \argmin_{ \b{x} \in \Omega} \mathbb{V}_{X_n \cup \{\b{x} \}}^\textsm{BQ}.
\end{equation*}
In higher dimensions there is little that we are able to say about qualitative properties of the resulting quadrature rules. However, when $d=1$ we can invoke Theorem \ref{thm:m-optimal-weights} since $X_n \cup \{x_{n+1}\}$ is locally $1$-optimal.

\begin{proposition} \label{prop:greedy-points} Suppose that $k \in C^\infty([a,b]^2)$ is totally positive of order 2. If $X_n \cup x_{n+1} \in \mathcal{S}^n$, then at least $\lfloor (n+3)/2 \rfloor$ of the weights of a $n+1$ point sequential Bayesian quadrature rule are positive.
\end{proposition}

\subsection{Other kernels and point sets} \label{sec:other-positive}

A number of combinations of kernels and point sets, that are not covered by the theory above, have been shown, either theoretically or experimentally, to yield positive Bayesian quadrature weights:
\begin{itemize}
\item \rev{The GP posterior mean for the Brownian motion kernel $k(x,x') = \min(x,x')$ on $[0,1]$ is a piecewise linear interpolant. As this implies that the Lagrange cardinal functions $u_{X,i}$ are non-negative, it follows from the identity $w_{X,i}^\textsm{BQ} = I_\nu(u_{X,i})$ that the weights are positive. See~\cite{Diaconis1988} and \cite[Lemma 8 in Section 3.2, Chapter 2]{Ritter2000} for more discussion.}
\item Suitably selected priors give rise to Bayesian quadrature rules whose posterior mean coincides with a classical rule, such a Gaussian quadrature~\cite{Karvonen2017a,Karvonen2018c}. Analysis of the weights and their positivity naturally reduces to that of the reproduced classical rule. 
\item There is convincing numerical evidence that the weights are positive if the nodes for the Gaussian kernel and measure on $\mathbb{R}$ are selected by suitable scaling the classical Gauss--Hermite nodes~\cite{Karvonen2018b}.
\item Uniform weighting (i.e., $w_{X,i}^\textsm{BQ} = 1/n$) can be achieved when certain quasi-Monte Carlo point sets and shift invariant kernels are used~\cite{Jagadeeswaran2018}.
\end{itemize}

\subsection{Upper bound on the sum of weights}\label{sec:positivity-sums}

We summarise below a simple yet generic result that has an important consequence on the stability of Bayesian quadrature in Section \ref{sec:stability}.
\begin{lemma} \label{lemma:weight-sum-generic}
Let $\Omega \subset \mathbb{R}^d$.  
If the Bayesian quadrature weights $w_{X,1}^\textsm{BQ},\dots,w_{X,n}^\textsm{BQ}$ are non-negative, then we have
\begin{equation*}
\sum_{i=1}^n w_{X,i}^\textsm{BQ}  \leq \frac{\sup_{\b{x} \in \Omega} I_\nu(k_{\b{x}})}{\inf_{\b{x},\b{x}' \in \Omega} k(\b{x},\b{x}')}.
\end{equation*}
\end{lemma}

\begin{proof}
The claim immediately follows from the property~\eqref{eq:BQ-exactness} that
$\sum_{i=1}^n w_{X,i}^\textsm{BQ} k_{\b{x}_j}(\b{x}_i) = I_\nu(k_{\b{x}_j})$ for each $j=1,\ldots,n$. \qed

\end{proof}

Combined with Corollary \ref{cor:optimal-weights}, we get a bound on the sum of absolute weights $ \sum_{i=1}^n |w_{X_n,i}^\textsm{BQ}|$, which is the main topic of discussion in Section \ref{sec:stability}.
\begin{corollary}
\label{cor:optimal-weights-sum} 
Let $\Omega = [a,b] \subset \mathbb{R}$.
If $k \in C^\infty([a,b]^2)$ is totally positive of order 2 and design points $X \in \mathcal{S}^n$ are locally optimal, then we have 
$$
 \sum_{i=1}^n |w_{X,i}^\textsm{BQ}| = \sum_{i=1}^n w_{X,i}^\textsm{BQ} \leq \frac{\sup_{x \in [a,b]} I_\nu(k_x)}{\inf_{x,x' \in [a,b]}  k(x,x')} . 
$$
\end{corollary}

Most importantly, Corollary \ref{cor:optimal-weights-sum} is applicable to the Gaussian kernel, for which the upper-bound is finite.
This result will be discussed in Section \ref{sec:Runge-phenomenon} in more detail.
One may see a supporting evidence in Fig.~\ref{fig:optimal-2d}, where the sum of weights seems to converge to a value around $1$.

\subsection{Higher dimensions} \label{sec:multi-positive}

As far as we are aware of, there are no extensions of the theory of Chebyshev systems to higher dimensions. Consequently, it is not possible to say much about positivity of the weights when $d > 1$. Some simple cases can be analysed, however.

Let $\Omega_1 = [a,b]$, $\nu_1$ be a measure on $\Omega_1$, $\Omega = \Omega_1^d \subset \mathbb{R}^d$ and $\nu = \nu_1^d$. That is, $\Omega =\Omega_1 \times \cdots \times \Omega_1$ and $\dif\nu(\b{x}) = \dif \nu_1(x_1) \times \cdots \times \dif \nu_1(x_d)$, where there are $d$ terms in the products.
Suppose that
\begin{enumerate}
\item[(i)] the point set $X$ is now a Cartesian product of one-dimensional sets $X_1 = \{x_1^1,\ldots,x_n^1\}\subset \Omega_1$: $X = X_1^d$;
\item[(ii)] the kernel is of product form: $k(\b{x},\b{x}') = \prod_{i=1}^d k_1(x_i,x_i')$ for some kernel $k_1$ on $\Omega_1$.
\end{enumerate}
A quadrature rule using Cartesian product points is called a \emph{tensor product rule}.
For such points the Bayesian quadrature weights $\b{w}_X^\textsm{BQ}$ are products of the one-dimensional weights $\b{w}_{X_1}^\textsm{BQ}$: the weight for the point $(x_{i(1)},\ldots,x_{i(d)}) \in X$ is $\prod_{j=1}^d w_{X_1,i(j)}^\textsm{BQ}$~\cite[Section 2.4]{Oettershagen2017}. 
In particular, if $k_1$ is totally positive and $X_1$ is a locally optimal set of points, then all the $n^d$ weights $\b{w}_X^\textsm{BQ}$ are positive.\footnote{Note that a tensor product rule based on an optimal one-dimensional point set need not be locally optimal for $\Omega$, $\nu$ and $k$.}
Analysis of more flexible sparse grid and symmetry based methods~\cite{Karvonen2018} might yield more interesting results.

We conclude this section with two numerical examples. Both of them involve the standard Gaussian measure 
\begin{equation*}
\dif\nu(\b{x}) = (2\pi)^{-d/2} \exp\bigg( - \frac{\norm[0]{\b{x}}^2}{2} \bigg) \dif \b{x} 
\end{equation*}
on $\Omega = \mathbb{R}^d$ and the Gaussian kernel~\eqref{eq:gauss-kernel}.

\paragraph{Locally optimal points} First, we investigated positivity of weights for locally optimal points. We set $\ell = 1$ and $d = 2$ and used a gradient-based \rev{quasi-Newton} optimisation method \rev{(MATLAB's \texttt{fminunc})} to find points that locally minimise the integral variance for $n=2,\ldots,20$. \rev{Optimisation was initialised with a set of random points. The point set output by the optimiser was then randomly perturbated and optimisation repeated for 20 times, each time initialising with the point set giving the smallest Bayesian quadrature variance so far. 
The weights were always computed directly from~\eqref{eq:BQ-weights}. 
However, to improve numerical stability, the kernel matrix $\b{K}_X$ was replaced by $\b{K}_X + 10^{-6} \b{I}$, where $\b{I}$ is the $n \times n$ identity matrix, during point optimisation.}
Some point sets generated using the same algorithm have appeared in \cite[Section IV]{Sarkka2016} (for other examples of optimal points in dimension two, see \cite{OHagan1992,Minka2000} and, in particular, \cite[Chapter 6]{Oettershagen2017}). 
The point sets we obtained appear sensible and all of them are associated with positive weights; four sets and their weights are depicted in Fig.\@~\ref{fig:optimal-2d}. \rev{For $n = 20$, the maximal value of a partial derivative of $\mathbb{V}_X^\textsm{BQ}$ at the computed points was $9 \times 10^{-10}$.}

\paragraph{Random points} Secondly, we investigated the validity of Theorem \ref{thm:BQ-positivity-general} in higher dimensions. We set $\ell = 1.5$ and $d = 4$ and counted the number of positive weights for $n=2,\ldots,1000$ when each $n$-point set is generated by drawing Monte Carlo samples from $\nu$.
Random samples are often used in Bayesian quadrature~\cite{RasmussenGhahramani2002,Briol2017,Briol2017a} and they also function as a suitable test case where structurality of point sets has little role in constraining behaviour of some subsets of the weights as happens when product or symmetric point designs are used.
Fig.\@~\ref{fig:gauss-weights} shows the proportion of positive weights; it appears that at least half of the weights for randomly drawn points are always positive. This supports the obvious conjectural extension to higher dimensions of Theorem \ref{thm:BQ-positivity-general}.

\begin{figure}[t!]
\centering
  \includegraphics{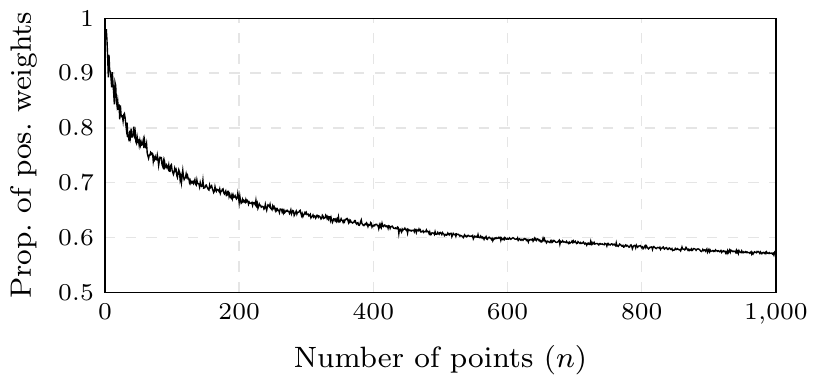}
  \caption{Proportion of positive weights for the Gaussian kernel and $n$ points drawn from the standard Gaussian distribution on $\mathbb{R}^4$. The results have been averaged over 50 independent Monte Carlo runs. Among all runs the minimal proportion encountered was exactly $1/2$.}\label{fig:gauss-weights}
\end{figure}

\section{Magnitudes of weights and the stability} \label{sec:stability}

This section studies the magnitudes of the  weights in a Bayesian quadrature rule and discusses how they are related to stability and robustness of the quadrature rule. 
We are in particular interested in the following quantity, which we call the Bayesian quadrature \emph{stability constant}:
\begin{equation} \label{eq:stability-const}
\Lambda_{X_n}^\textsm{BQ} \coloneqq \sum_{i=1}^n \abs[0]{ w_{X_n,i}^\textsm{BQ} }.
\end{equation}
To make dependency on $n$ more explicit, the quadrature point set is denoted by $X_n$ instead of $X$ in this section.
The terminology is motivated by the close connection of $\Lambda_{X_n}^\textsm{BQ}$ to the \emph{Lebesgue constant} $\Lambda_{X_n}$, a quantity that characterises the stability of an interpolant. For kernel interpolants the Lebesgue constant is
\begin{equation*}
\Lambda_{X_n} \coloneqq \sup_{\b{x} \in \Omega} \sum_{i=1}^n \abs[0]{u_{X_n,i}(\b{x})},
\end{equation*}
where $u_{X_n,i}$ are Lagrange cardinal functions from Section \ref{sec:bq-basics}.
The connection to \eqref{eq:stability-const} arises from the fact that $w_{X_n,i}^\textsm{BQ} = I_\nu(u_{X_n,i})$ for $i=1,\dots,n$.

\rev{The importance of the stability constant~\eqref{eq:stability-const} is illustrated by the following argument. 
Let $\mu_f^*$ be an optimal approximant to the integrand function $f \colon \Omega \to \mathbb{R}$ in the span of $\{ k_{\b{x}_i}\}_{i=1}^n$ in the sense that
\begin{equation*}
    \mu_{f}^* \in \argmin_{ \mu_f \in \mathrm{span} \{ k_{\b{x}_i}\}_{i=1}^n} \, \norm[0]{f-\mu_f}_\infty,
\end{equation*}
where $\norm[0]{f-\mu_f}_\infty \coloneqq \sup_{ \b{x} \in \Omega} \abs[0]{f(\b{x}) - \mu_f(\b{x})}$ is the uniform norm.
Note that $\mu_f^*$ does not in general interpolate $f$ at $X_n$ nor coincide with the Gaussian process posterior mean $\mu_{X,f}$. Then
\begin{equation*}
    \begin{split}
        \abs[0]{I_\nu(f) - I_{X_n}^\textsm{BQ}(f)} \hspace{-2cm}& \\
        &\leq \abs[0]{I_\nu(f) - I_\nu(\mu_f^*)} + \abs[0]{I_\nu(\mu_f^*) - I_{X_n}^\textsm{BQ}(f)} \\
        &= \abs[0]{I_\nu(f) - I_\nu(\mu_f^*)} + \abs[0]{I_{X_n}^\textsm{BQ}(\mu_f^*) - I_{X_n}^\textsm{BQ}(f)} \\
        &\leq \norm[0]{f-\mu_f^*}_\infty + \sum_{i=1}^n \abs[0]{w_{X_n,i}^\textsm{BQ}} \abs[0]{\mu_f^*(\b{x}_i) - f(\b{x}_i)} \\
        &\leq (1 + \Lambda_{X_n}^\textsm{BQ}) \norm[0]{f-\mu_f^*}_\infty,
    \end{split}
\end{equation*}
where we have used the fact that $I_{X_n}^\textsm{BQ}(g) = I_\nu(g)$ if $g \in \mathrm{span} \{ k_{\b{x}_i}\}_{i=1}^n$. That is, the approximation error by a Bayesian quadrature rule can be related to that by the best uniform approximant via the stability constant.
}
\rev{The stability constant also controls the error introduced by inaccurate funtion evaluations.}
Suppose that the function evaluations contain  errors (which may be numerical or stochastic), denoted by $\epsilon_i$ and modelled as indendent zero-mean random variables with variance $\sigma^2$. 
Then the mean-square error (where the expectation is w.r.t.~$\epsilon_1,\dots,\epsilon_n$) of Bayesian quadrature is given by
\begin{equation*}
\begin{split}
\mathbb{E} &\Bigg[ \bigg( I_\nu(f) - \sum_{i=1}^n  w_{X_n,i}^\textsm{BQ} [f(\b{x}_i) + \epsilon_i] \bigg)^2 \Bigg] \\
&= \bigg( I_\nu(f) - \sum_{i=1}^n  w_{X_n,i}^\textsm{BQ} f(\b{x}_i) \bigg)^2 + \sigma^2 \sum_{i=1}^n ( w_{X_n,i}^\textsm{BQ})^2 \\
&\leq \bigg( I_\nu(f) - \sum_{i=1}^n  w_{X_n,i}^\textsm{BQ} f(\b{x}_i) \bigg)^2 + \sigma^2 \bigg(\sum_{i=1}^n \abs[0]{ w_{X_n,i}^\textsm{BQ}} \bigg)^2.
\end{split}
\end{equation*}
This implies a small stability constant \eqref{eq:stability-const} suppresses the additional error caused by the perturbations $\epsilon_i$.
A third motivating example will be given in Section \ref{sec:sobolev-convergence}, after introducing necessary notation.

\rev{It is clear from Lemma~\ref{lemma:weight-sum-generic} that if the weights are positive for every $n$, the stability constant remains uniformly bounded. However, the results on positivity in the preceding section are valid only when $d = 1$ and the kernel is totally positive. This section uses a different technique to analyse the stability constant.}
The results are based on those in \cite{DeMarchiSchaback2010}, which are applicable to kernels that induce Sobolev-equivalent RKHSs (e.g., Mat\'{e}rn kernels). Accordingly, we mainly focus on such kernels in this section.
We begin by reviewing basic properties of Sobolev spaces in Section \ref{sec:Sobolev} and convergence results for Bayesian quadrature in Section \ref{sec:sobolev-convergence}.
The main results, Theorem \ref{thm:bound-weight} and Corollary \ref{cor:stability}, on the magnitudes of quadrature weights and the stability constant appear in Section \ref{sec:stability-sobolev}.
We discuss a relevant stability issue, known as the Runge phenomenon, for infinitely smooth kernels such as the Gaussian kernel in Section \ref{sec:Runge-phenomenon}.
Finally, simulation results in Section \ref{sec:stability-examples} demonstrate that the obtained upper bound is conservative; there is much room for improving the results.

\paragraph{Notation and basic definitions.} 

The Fourier transform $\hat{f}$ of a Lebesgue integrable $f \colon \mathbb{R}^d \to \mathbb{R}$ is defined by
\begin{equation*}
    \hat{f}(\b{\xi}) \coloneqq (2\pi)^{-d/2} \int_{\mathbb{R}^d} f(\b{x}) \mathrm{e}^{-\sqrt{-1}\ \b{\xi}^\T \b{x}} \dif \b{x}, \quad \b{\xi} \in \mathbb{R}^d.
\end{equation*}
Two normed vector spaces $\mathcal{F}_1$ and $\mathcal{F}_2$ are {\em norm-equivalent} if $\mathcal{F}_1 = \mathcal{F}_2$ as a set and there exist constants $C_1, C_2 > 0$ such that
\begin{equation*}
C_1 \norm[0]{f}_{\mathcal{F}_2} \leq \norm[0]{f}_{\mathcal{F}_1} \leq C_2 \norm[0]{f}_{\mathcal{F}_2} \quad  \text{for all} \quad f \in \mathcal{F}_1.
\end{equation*}

\subsection{Kernels inducing Sobolev-equivalent RKHSs} \label{sec:Sobolev}

\rev{
Let $\Phi \colon \mathbb{R}^d \to \mathbb{R}$ be a continuous and integrable positive-definite function with Fourier transform satisfying
\begin{equation} \label{eq:kernel-cond-sobolev}
c_1 (1+\norm[0]{\b{\xi}}^2)^{-r} \leq  \hat{\Phi}(\b{\xi}) \leq c_2 (1+\norm[0]{\b{\xi}}^2)^{-r}
\end{equation}
for $r > d/2$, some positive constants $c_1$ and $c_2$, and for all $\b{\xi} \in \mathbb{R}^d$.
In this section, we consider shift-invariant kernels on $\mathbb{R}^d$ of the form $k(\b{x},\b{x}') = \Phi( \b{x} - \b{x}')$.
For instance, a Mat\'{e}rn kernel \cite[Section 4.2.1]{Rasmussen2006}
\begin{equation*}\label{eq:matern-kernel}
k_\rho(\b{x},\b{x}') = \frac{2^{1-\rho}}{\Gamma(\rho)} \bigg( \frac{\sqrt{2\rho}\norm[0]{\b{x}-\b{x}'}}{\ell} \bigg)^\rho \mathrm{K}_\rho\bigg( \frac{\sqrt{2\rho}\norm[0]{\b{x}-\b{x}'}}{\ell} \bigg)
\end{equation*}
with smoothness parameter $\rho := r - d/2$ and length-scale parameter $\ell > 0$ satisfies \eqref{eq:kernel-cond-sobolev}.\footnote{\rev{Note that the smoothness parametrisation $\rho = r$ is often used. With this parametrisation $k_\rho$ would satisfy~\eqref{eq:kernel-cond-sobolev} with the exponent $-(r+d/2)$ and its RKHS would be norm-equivalent to $H^{r+d/2}(\mathbb{R}^d)$.}} Here $\mathrm{K}_\rho$ is the modified Bessel function of the second kind of order~$\rho$. 
Another notable class of kernels satisfying \eqref{eq:kernel-cond-sobolev} are Wendland kernels  \cite[Theorem 10.35]{Wendland2005}.
}

By \cite[Corollary 10.13]{Wendland2005}, the RKHS $\mathcal{H}(k)$ of any kernel $k$ satisfying \eqref{eq:kernel-cond-sobolev} is norm-equivalent to the {\em Sobolev space} $H^r(\mathbb{R}^d)$ of order $r > d/2$ on $\mathbb{R}^d$, which is a Hilbert space consisting of square-integrable and continuous functions $f \colon \mathbb{R}^d \to \mathbb{R}$ such that
\begin{equation*}
    \| f \|_{H^r(\mathbb{R}^d)}^2 := (2\pi)^{-d/2} \int_{\mathbb{R}^d} \big( 1 + \norm[0]{\b{\xi}}^2 \big)^r \abs[0]{\hat{f}(\b{\xi})}^2 \dif \b{\xi} < \infty.
\end{equation*} 
As can be seen from this expression, $r$ quantifies the smoothness of functions in $H^r(\mathbb{R}^d)$: as $r$ increases, function in $H^r(\mathbb{R}^d)$ become smoother.

The Sobolev space $H^r(\Omega)$ on a general measurable domain $\Omega \subset \mathbb{R}^d$ can be defined as the restriction of $H^r(\mathbb{R}^d)$ onto $\Omega$.
The kernel $k$ satisfying \eqref{eq:kernel-cond-sobolev}, when seen as a kernel on $\Omega$, then induces an RKHS that is norm-equivalent to $H^r (\Omega)$ \cite[Theorems 10.12, 10.46 and 10.47]{Wendland2005}.\footnote{The reader may ask whether $\Omega$ needs to have a Lipschitz boundary for this norm-equivalence, but this assumption is indeed not needed. 
The assumption that $\Omega$ has a Lipschitz boundary is required when using {\em Stein's extension theorem} \cite[p. 181]{Ste70} for Sobolev spaces defined using weak derivatives (see the proof of \cite[Corollary 10.48]{Wendland2005}). 
On the other hand, we consider here a Sobolev space defined in terms of the Fourier transform, and the norm-equivalence follows from the extension and restriction theorems for a generic RKHS \cite[Theorems 10.46 and 10.47]{Wendland2005} and the expression of the RKHS norm in terms of Fourier transforms \cite[Theorem 10.12]{Wendland2005}.}

\subsection{Convergence for Sobolev-equivalent kernels} \label{sec:sobolev-convergence}

Recall from Section \ref{sec:RKHS} that integration error by a Bayesian quadrature rule for functions in $\mathcal{H}(k)$ satisfies
\begin{equation*}
\abs[1]{I_\nu(f) - I_{X_n}^\textsm{BQ}(f)} \leq \norm[0]{f}_{\mathcal{H}(k)} e_{\mathcal{H}(k)}({X_n}, \b{w}_X^\textsm{BQ}),
\end{equation*}
so that in convergence analysis only the behaviour of the worst-case error needs to be considered.
If the RKHS is norm-equivalent to a Sobolev space, rates of convergence for Bayesian quadrature can be established. These results follow from~\cite[Corollary~4.1]{Arcangeli2007}. See~\cite[Corollary 11.33]{Wendland2005} or~\rev{\cite[Proposition~3.6]{WendlandRieger2005}} for earlier and slightly more restricted results that require $\lfloor r \rfloor > d/2$ and~\cite[Proposition 4]{Kanagawa2018} for a version specifically for numerical integration. Some assumptions, satisfied by all domains of interest to us, are needed; see for instance \cite[Section~3]{Kanagawa2018} for precise definitions. 

\begin{assumption} \label{ass:Omega}  The set $\Omega \subset \mathbb{R}^d$ is a bounded open set that satisfies an interior cone condition and has a Lipschitz boundary.
\end{assumption}

This assumption essentially says that the boundary of $\Omega$ is sufficiently regular (Lipschitz boundary) and that there is no ``pinch point'' on the boundary of $\Omega$ (interior cone condition).
Convergence results are expressed in terms of the \emph{fill-distance}
\begin{equation*}
h_{X_n,\Omega} \coloneqq \sup_{\b{x} \in \Omega} \min_{i=1,\ldots,n} \norm[0]{\b{x}-\b{x}_i}
\end{equation*}
that quantifies the size of the largest ``hole'' in an $n$-point set $X_n$. We use $\lesssim$ to denote an inequality that is valid up to a constant independent of $n$, number of points, and $f$, the integrand. That is, for generic sequences of functionals $g_n$ and $h_n$, $g_{n}(f) \lesssim h_{n}(f)$ means that there is a constant $C > 0$ such that $g_{n}(f) \leq C h_n(f)$ for all $n \in \mathbb{N}$ and any $f$ in a specified class of functions.

\begin{theorem} \label{thm:convergence} Suppose that (i) $\Omega$ satisfies Assumption \ref{ass:Omega} (ii) that the measure $\nu$ has a bounded (Lebesgue) density function and that (iii) the kernel $k$ satisfies \eqref{eq:kernel-cond-sobolev} for a constant $r$ such that $r > d/2$. Then
\begin{equation*}
\abs[1]{I_\nu(f) - I_{X_n}^\textsm{BQ}(f)} \lesssim \norm[0]{f}_{H^{r}(\Omega)} h_{X_n,\Omega}^r
\end{equation*}
for any $f \in H^{r}(\Omega)$ when the fill-distance is sufficiently small.
\end{theorem}

The following simple result is an immediate consequence of this theorem.

\begin{proposition} \label{prop:sum-to-one} Suppose that the assumptions of Theorem \ref{thm:convergence} are satisfied. Then $\abs[0]{1 - \sum_{i=1}^n w_{X_n,i}^\textsm{BQ}} \lesssim h_{X_n,\Omega}^r$  when the fill-distance is sufficiently small.
\end{proposition}
\begin{proof} Under the assumptions, constant functions are in $H^{r}(\Omega)$. Setting $f \equiv 1$ in~\eqref{eq:quasi-unif-conv} verifies the claim. \qed
\end{proof}

\rev{Note that the same argument can be used whenever a general rate of convergence for functions in an RKHS is known and constant functions are contained in the RKHS. However, this is not always the case; for example, the RKHS of the Gaussian kernel~\eqref{eq:gauss-kernel} does not contain polynomials~\cite[Theorem~2]{Minh2010}.}

Rates explicitly dependent on the number of points are achieved for point sets that are \emph{quasi-uniform}, which is to say that
\begin{equation*} \label{eq:quasi-uniform}
\tilde{c}_1 q_{X_n} \leq h_{X_n,\Omega} \leq \tilde{c}_2 q_{X_n}
\end{equation*}
for some constants $\tilde{c}_1, \tilde{c}_2 > 0$ independent of $n$. Here
\begin{equation*}
q_X \coloneqq \frac{1}{2} \min_{i \neq j} \norm[0]{\b{x}_i - \b{x}_j}
\end{equation*}
is the \emph{separation distance}. In dimension $d$ quasi-uniform sets satisfy $h_{X_n,\Omega} = \mathcal{O}(n^{-1/d})$ as $n \to \infty$ (e.g., regular product grids). In Theorem~\ref{thm:convergence} we thus obtain the rate
\begin{equation}\label{eq:quasi-unif-conv}
\abs[0]{I_\nu(f) - I_{X_n}^\textsm{BQ}(f)} \lesssim n^{-r/d}
\end{equation}
for $f \in H^{r}(\Omega)$ when the point sets are quasi-uniform and $n$ is sufficiently large. 

Of course, it is the stability constant $\Lambda_{X_n}^\textsm{BQ} = \sum_{i=1}^n \abs[0]{w_{X_n,i}^\textsm{BQ}}$ that we analyse next whose behaviour is typically more consequential. 
However, the above proposition may be occasionally interesting if one desires to interpret Bayesian quadrature as a weighted Dirac approximation $\nu_\textsm{BQ} \coloneqq \sum_{i=1}^n w_{X_n,i}^\textsm{BQ} \delta_{\b{x}_i} \approx \nu$ of a probability measure (i.e., $\nu_\textsm{BQ}(\Omega) \approx 1$). 
Note that there is also a simple way to ensure summing up to one of the weights by inclusion of a non-zero prior mean function for the Gaussian process prior; see~\cite{OHagan1991} and \cite[Section 2.3]{Karvonen2018c}

Finally, we provide a third example that highlights the importance of analysing the stability constant.
Kanagawa et al. \cite[Section 4.1]{Kanagawa2018} (see also~\cite{Kanagawa2016}) studied convergence rates of kernel-based quadrature rules in Sobolev spaces when the integrand is potentially {\em rougher} (i.e., $f \in H^{s}(\Omega)$ for some $s \leq r$) than assumed.
If $s < r$, the integrand $f$ may not belong to the Sobolev space $H^{r}(\Omega)$ that is assumed by the user when constructing the quadrature rule; therefore this is a {\em misspecifed} setting.
Under certain conditions, they showed~\cite[Corollary 7]{Kanagawa2018} that if $\Lambda_{X_n}^\textsm{BQ} \lesssim n^c$ for a constant $c \geq 0$, then
\begin{equation} \label{eq:kanagawa-conv}
\abs[1]{I_\nu(f) - I_{X_n}^\textsm{BQ}(f)} \lesssim n^{-s/d + c(r-s)/r},
\end{equation}
when $X_n$ are quasi-uniform.

The condition $\Lambda_{X_n}^\textsm{BQ} \lesssim n^c$ means that the stability constant $\Lambda_{X_n}^\textsm{BQ}$ should not grow quickly as $n$ increases.
The bound \eqref{eq:kanagawa-conv} shows that the error in the misspecified setting becomes small if $c$ is small. 
This implies that if the stability constant $\Lambda_{X_n}^\textsm{BQ}$ does not increase quickly, then the quadrature rule becomes robust against the misspecification of a prior.
This provides a third motivation for understanding the behaviour of $\Lambda_{X_n}^\textsm{BQ}$.

\subsection{Upper bounds for absolute weights} \label{sec:stability-sobolev}

We now analyse magnitudes of individual weights and the stability constant \eqref{eq:stability-const}. 
We first derive an upper bound on the magnitude of each weight $w_{X_n,i}^\textsm{BQ}$.
The proof of this result is based on an upper bound on the $L^2(\Omega)$ norm of Lagrange functions derived in~\cite{DeMarchiSchaback2010}.

\begin{theorem}  \label{thm:bound-weight}
Suppose that (i) $\Omega$ satisfies Assumption \ref{ass:Omega}, that (ii) the measure $\nu$ has a bounded (Lebesgue) density function and that (iii) the kernel $k$ satisfies \eqref{eq:kernel-cond-sobolev} for a constant $r$ such that $ r > d/2$.
Then
\begin{equation} \label{eq:bound-abs-weight}
    \abs[0]{ w_{X_n,i}^\textsm{BQ} } \lesssim \bigg( \frac{h_{X_n,\Omega}}{q_{X_n}}\bigg)^{r-d/2} h_{X_n,\Omega}^{d/2}
\end{equation}
for all $i=1,\dots,n$, provided that $h_{X_n,\Omega}$ is sufficiently small. 
When $X_n$ are quasi-uniform, this becomes
\begin{equation} \label{eq:bound-weight-quasi}
\abs[0]{ w_{X_n,i}^\textsm{BQ} } \lesssim n^{-1/2}
\end{equation}
for $n$ large enough.
\end{theorem}
\begin{proof}
It is proved in \cite[Theorem 1]{DeMarchiSchaback2010} that each of the Lagrange functions $u_{X_n,i}$ admits the bound
\begin{equation} \label{eq:lagrange-L2-bound}
\bigg( \int_\Omega u_{X_n,i}(\b{x})^2 \dif \b{x} \bigg)^{1/2} \lesssim \bigg( \frac{h_{X_n,\Omega}}{q_{X_n}}\bigg)^{r-d/2} h_{X_n,\Omega}^{d/2},
\end{equation}
provided that $h_{X_n,\Omega}$ is sufficiently small.
Let $\norm[0]{\nu}_\infty < \infty$ stand for the supremum of the density function of $\nu$. 
Then it follows from $w_{X_n,i}^\textsm{BQ} = I_\nu(u_{X_n,i})$ that
\begin{equation*}
\begin{split}
|w_{X_n,i}^\textsm{BQ}| &\leq \int |u_{X_n,i} (\b{x})| d\nu(\b{x}) \\
&\leq \left(\int_\Omega u_{X_n,i}(\b{x})^2 \dif\nu( \b{x} )  \right)^{1/2} \\
& \leq \norm[0]{\nu}_\infty \left( \int_\Omega u_{X_n,i}(\b{x})^2 \dif \b{x} \right)^{1/2}.
\end{split}
\end{equation*}
Inequality \eqref{eq:bound-abs-weight} now follows from \eqref{eq:lagrange-L2-bound}.
When $X_n$ are quasi-uniform, the ratio $h_{X_n,\Omega}/q_{X_n}$ remains bounded and $h_{X_n,\Omega}$ behaves like $n^{-1/d}$. \qed 

\end{proof}

An important consequence of Theorem \ref{thm:bound-weight} is that the magnitudes of quadrature weights decrease uniformly to zero as $n$ increases if the design points are quasi-uniform and $\nu$ has a density.
In other words, none of the design points will have a constant weight that does not decay.
This is similar to importance sampling, where the weights decay uniformly at rate $1/n$.
As a direct corollary of Theorem \ref{thm:bound-weight} we obtain bounds on the stability constant $\Lambda_{X_n}^\textsm{BQ}$.

\begin{figure}[t]
\centering
  \includegraphics{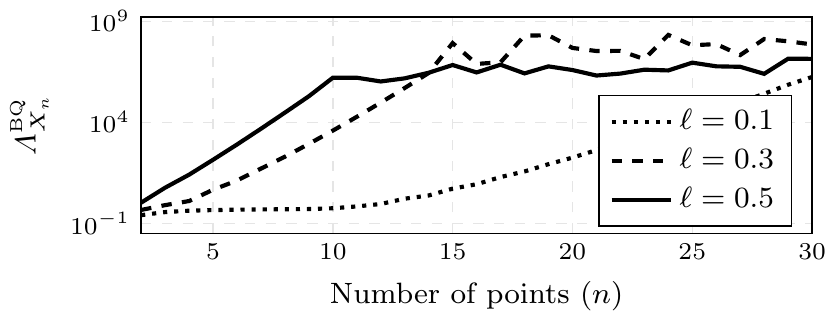}
  \caption{Bayesian quadrature stability constants for the Gaussian kernel~\eqref{eq:gauss-kernel} with different length-scales, the uniform measure on $[0,1]$ and $n$ points uniformly placed on this interval (end points not included). The levelling off appears to be caused by loss of numerical precision.
  }\label{fig:stability-gauss}
\end{figure}

\begin{corollary} \label{cor:stability} 
Under the assumptions of Theorem \ref{thm:bound-weight} and provided that $h_{X_n,\Omega}$ is sufficiently small we have
\begin{equation} \label{eq:stability-sobolev}
\Lambda_{X_n}^\textsm{BQ} \lesssim n \bigg( \frac{h_{X_n,\Omega}}{q_{X_n}}\bigg)^{r-d/2} h_{X_n,\Omega}^{d/2}.
\end{equation}
When $X_n$ are quasi-uniform and $n$ sufficiently large this becomes
\begin{equation} \label{eq:stability-sobolev-uniform}
\Lambda_{X_n}^\textsm{BQ} \lesssim \sqrt{n}.
\end{equation}
\end{corollary}

While the bounds of Corollary~\ref{cor:stability} are somewhat conservative (as will be demonstrated in Section \ref{sec:stability-examples}), they are still useful in understanding the factors affecting stability and robustness of Bayesian quadrature. 
That is, inequality \eqref{eq:stability-sobolev} shows that the stability constant can be made small if the ratio $h_{X_n,\Omega} / q_{X_n}$ is kept small; this is possible if the point set is sufficiently uniform.

Another important observation concerns the exponent ${r-d/2}$ of the ratio $h_{X_n,\Omega} / q_{X_n}$: if the smoothness $r$ of the kernel is large, then the stability constant may also become large if the points are not quasi-uniform.
This is true because $h_{X_n,\Omega} / q_{X_n} \geq 1$ for any configuration of $X_n$, as can be seen easily from the definitions of $q_{X_n}$ and $h_{X_n,\Omega}$.
This observation implies that the use of a smoother kernel may lead to higher numerical instability.
Accordingly, we next discuss stability of infinitely smooth kernels and the Runge phenomenon that manifests itself in this setting.

\begin{figure*}[th!]
\centering
  \includegraphics{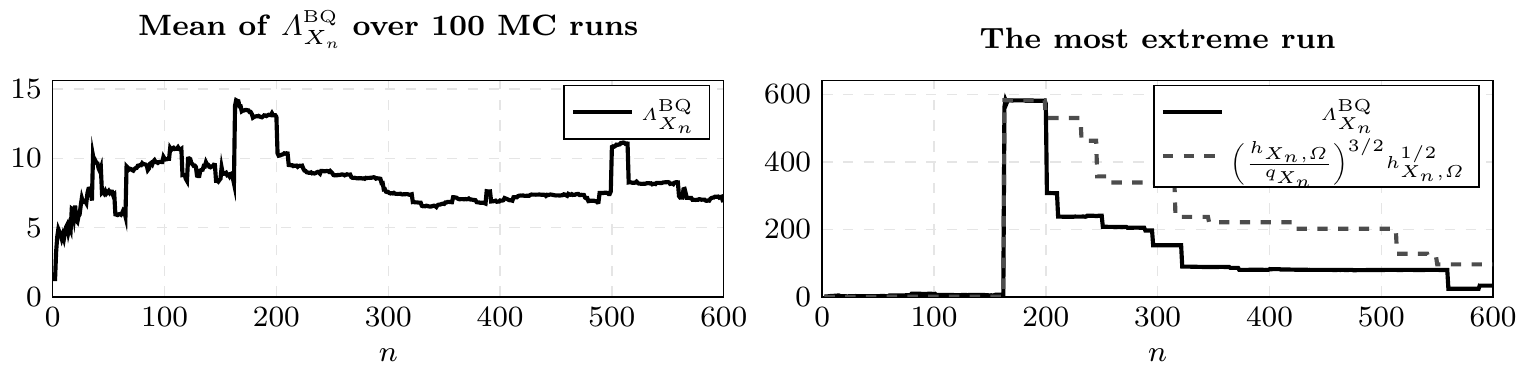}
  \caption{The Bayesian quadrature stability constant for a Mat\'{e}rn kernel, the uniform measure on $[0,1]$ and $n$ points drawn from the uniform distribution. \emph{Left}: $\Lambda_{X_n}^\textsm{BQ}$ averaged over 100 independent Monte Carlo runs. \emph{Right}: the run where most extreme behaviour, in terms of $\Lambda_{X_n}^\textsm{BQ}$ attaining maximal value, was observed. Plotted are both $\Lambda_{X_n}^\textsm{BQ}$ and a scaled version of $(h_{X_n,\Omega}/q_{X_n})^{3/2} h_{X_n,\Omega}^{1/2}$ (its true maximum was roughly $2.2 \times 10^7$) that is expected to control the stability constant. Note that the theoretical upper bound~\eqref{eq:stability-sobolev} contains an additional multiplication by $n$.}\label{fig:matern-random}
\end{figure*}

\subsection{On infinitely smooth kernels} \label{sec:Runge-phenomenon}

While the theoretical results of this section only concern kernels of finite smoothness, we make a few remarks on the stability of Bayesian quadrature when using infinitely smooth kernels, such as the Gaussian kernel. 
When using such a kernel, Bayesian quadrature rules suffer from the famous {\em Runge phenomenon}: if equispaced points are used, then Lebesgue constants and the stability constants grow rapidly; see ~\cite[Section 4.3]{Oettershagen2017} and \cite{Platte2005,Platte2011}. 
This effect is demonstrated in Fig. \ref{fig:stability-gauss}, \rev{and can be seen also in~\cite[Table~1]{SommarivaVianello2006b}}.

A key point is that Runge phenomenon typically occurs when the design points are quasi-uniform (e.g., equispaced).
This means that quasi-uniformity of the points does not ensure stability of Bayesian quadrature when the kernel is infinitely smooth. Care has to be taken if a numerically stable Bayesian quadrature rule is to be constructed with such a kernel. One possibility is to use locally optimal design points from Section \ref{sec:positivity-optimal}. Corollary \ref{cor:optimal-weights-sum} then guarantees uniform boundedness of the stability constant, at least when $d=1$.

\subsection{A numerical example} \label{sec:stability-examples}

Numerical examples of the behaviour of kernel Lebesgue constants can be found in~\cite{DeMarchiSchaback2008}, where it was observed that the theoretical bounds similar to~\eqref{eq:stability-sobolev-uniform} are conservative: the Lebesgue constant appears to remain uniformly bounded. Bayesian quadrature weights are no different. We experimented with the Mat\'{e}rn kernel
\begin{equation*}
k_{3/2}(x,x') = \bigg(1 + \frac{\sqrt{3}\abs[0]{x-x'}}{\ell} \bigg) \exp\bigg( -\frac{\sqrt{3}\abs[0]{x-x'}}{\ell}\bigg)
\end{equation*}
with length-scale $\ell = 0.5$ and the uniform measure on the interval $[0,1]$. When uniformly spaced points were used, all weights remained positive and their sum quickly converged to one when $n$ was increased. In contrast, Corollary \ref{cor:stability} provides the, up to a constant, upper bound $\sqrt{n}$ that is in this case clearly very conservative. 
When points were drawn from the uniform distribution on $[0,1]$, more interesting behaviour was observed (Fig. \ref{fig:matern-random}). As expected, the magnitude of $\Lambda_{X_n}^\textsm{BQ}$ was closely related to the ratio $h_{X_n,\Omega}/q_{X_n}$. Nevertheless, increase in $n$ did not generally correspond to increase in $\Lambda_{X_n}^\textsm{BQ}$.

\rev{Note that the results of Section~\ref{sec:positivity} do not explain why the weights became positive in this experiment, because Matérn kernels do not appear to be totally positive \emph{even if the differentiability requirements were to be relaxed and only single zeroes counted} (recall Remark~\ref{remark:chebyshev}).
We have numerically observed that selecting $n > \rho + 1/2$ and point sets such that $\max X < \min Y$ makes the matrix $\b{K}_{Y,X}$ discussed in Section~\ref{sec:totally-positive} singular for the Mat\'ern kernel $k_\rho$.
This implies that there is a non-trivial linear combination of the $n$ Mat\'ern translates at $X$ that vanishes at more than $n-1$ points. 
When $\rho = 1/2$ and $\ell = 1$ (so that $k_{\rho}(x,x') = \mathrm{e}^{-\abs[0]{x-y}}$), an analytical counterexample can be constructed by setting $X = \{x_1, x_2\}$ and $Y = \{x_1 + h, x_2 + h\}$ with $h > x_2 - x_1$. Then
\begin{equation*}
    \b{K}_{Y,X} = \mathrm{e}^{-h} \begin{bmatrix} 1 & \mathrm{e}^{-(x_2 - x_1)} \\ \mathrm{e}^{-(x_1-x_2)} & 1 \end{bmatrix},
\end{equation*}
which is not invertible because multiplying the first row by $\mathrm{e}^{-(x_1-x_2)}$ yields the second row.
Therefore the positivity of quadrature weights for Mat\'erns and other kernels with finite smoothness requires a further research.
}

\begin{acknowledgements}
TK was supported by the Aalto ELEC Doctoral School.
MK acknowledges support by the European Research Council (StG Project PANAMA).
SS was supported by the Academy of Finland project 313708.

This material was developed, in part, at the \textit{Prob Num 2018} workshop hosted by the Lloyd's Register Foundation programme on Data-Centric Engineering at the Alan Turing Institute, UK, and supported by the National Science Foundation, USA, under Grant DMS-1127914 to the Statistical and Applied Mathematical Sciences Institute. 
Any opinions, findings, conclusions or recommendations expressed in this material are those of the author(s) and do not necessarily reflect the views of the above-named funding bodies and research institutions.
\end{acknowledgements}

\bibliographystyle{spmpsci}      

\providecommand{\BIBYu}{Yu}

\end{document}

%% file: plotting.tex
\usepackage{tikz}
\usepackage{pgfplots}

\usepackage{subcaption}
\pgfplotsset{compat=newest}
\pgfplotsset{plot coordinates/math parser=false}

\usetikzlibrary{decorations.pathreplacing}
\usetikzlibrary{patterns}
\usepgfplotslibrary{groupplots}

\newlength{\figwidth}
\newlength{\figheight}
\definecolor{gray1}{gray}{0.0}
\definecolor{gray2}{gray}{0.25}
\definecolor{gray3}{gray}{0.5}
\definecolor{gray4}{gray}{0.7}
\definecolor{gray5}{gray}{0.9}

\newlength{\prel}

\pgfplotsset{
  title style = {font=\small},
}